\documentclass[11pt,a4paper]{article}
\usepackage{amsthm}	
\usepackage{amsfonts, amssymb, amsmath, amscd, latexsym}
\usepackage{eucal, enumerate, latexsym}
\usepackage{graphics}
\usepackage{graphicx}
\usepackage{mathscinet}
\usepackage{psfrag}
\usepackage{tikz}
\usepackage{fancyhdr}
\usetikzlibrary{decorations.markings}

\newcommand{\R}{\mathbb R}

\newtheorem{theorem}{Theorem}[section]
\newtheorem{corollary}{Corollary}[section]
\newtheorem{proposition}{Proposition}[section]

\newtheorem{remark}{Remark}[section]
\newtheorem{lemma}{Lemma}[section]

\newcommand{\ds}{\displaystyle}

\renewcommand{\epsilon}{\varepsilon}

\renewcommand{\phi}{\varphi}
\newenvironment{proofof}[1]{\smallskip\noindent\emph{Proof of #1.}%
\hspace{1pt}}{\hspace{-5pt}{\nobreak\quad\nobreak\hfill\nobreak%
$\square$\vspace{8pt}\par}\smallskip\goodbreak}

\renewcommand{\c}{s} 

\makeatletter

\newlength{\captionwidth}
\long\def\@makecaption#1#2{%
   \vskip 10\p@
   \setbox\@tempboxa\hbox{#1: #2}%
   \ifdim \wd\@tempboxa > \captionwidth 
       \hbox to\hsize{\hfil
       \parbox[t]{\captionwidth}{
       \small#1: \small#2\par}
       \hfil}
     \else
       \hbox to\hsize{\hfil\box\@tempboxa\hfil}%
   \fi}

\makeatother

\usetikzlibrary{decorations.pathreplacing,decorations.markings}
\tikzset{
  on each segment/.style={
    decorate,
    decoration={
      show path construction,
      moveto code={},
      lineto code={
        \path [#1]
        (\tikzinputsegmentfirst) -- (\tikzinputsegmentlast);
      },
      curveto code={
        \path [#1] (\tikzinputsegmentfirst)
        .. controls
        (\tikzinputsegmentsupporta) and (\tikzinputsegmentsupportb)
        ..
        (\tikzinputsegmentlast);
      },
      closepath code={
        \path [#1]
        (\tikzinputsegmentfirst) -- (\tikzinputsegmentlast);
      },
    },
  },
  mid arrow/.style={postaction={decorate,decoration={
        markings,
        mark=at position .6 with {\arrow[#1]{stealth}}
      }}},
}

\begin{document}

\title{
  The role of convection in the existence of wavefronts for biased movements}
\author{
Diego Berti\footnote{Department of Mathematics, University of Pisa, Italy}
\and
Andrea Corli\footnote{Department of Mathematics and Computer Science, University of Ferrara, Italy}
\and
Luisa Malaguti\footnote{Department of Sciences and Methods for Engineering, University of Modena and Reggio Emilia, Italy}
}

\maketitle

\begin{abstract}
We investigate a model, inspired by Johnson et al. (2017), see \cite{JBMS}, to describe the movement of a biological population which consists of isolated and grouped organisms. We introduce biases in the movements and then obtain a scalar reaction-diffusion equation which includes a convective term as a consequence of the biases. We focus on the case the diffusivity makes the parabolic equation of forward-backward-forward type and the reaction term models a strong Allee effect, with the Allee parameter lying between the two internal zeros of the diffusion. In such a case, the unbiased equation (i.e., without convection) possesses no smooth traveling-wave solutions; on the contrary, in the presence of convection, we show that traveling-wave solutions do exist for some significant choices of the parameters. We also study the sign of their speeds, which provides information on the long term behavior of the population, namely, its survival or extinction.
\end{abstract}

\vspace{1cm}


\section{Introduction}
\label{sec:intro}
In this paper we investigate a model to describe the movement of biological organisms. Its detailed presentation appears in Section \ref{s:model}. Inspired by the recent paper \cite{JBMS}, we assume that the population is constituted of isolated and grouped organisms; our discussion  is presented in the case of a single spatial dimension but could be extended to the whole space. The first rigorous mathematical deduction of movement for organisms appeared in \cite{Patlak}; since then, several models have been proposed, see for instance \cite{HPO, JBMS, Murray, Okubo-Levin, OS} and references there. In this context, a common procedure is to start from a discrete framework where the transition probabilities per unit time $\tau$ and for a one-step jump-width $l$ are assigned, and then pass to the limit for $\tau,l\to0$. In the aforementioned papers the limiting assumptions make the diffusivity totally  responsible for the movement, and no convection term appears; see however \cite[\S 5.3]{Okubo-Levin} and \cite{Padron}, for instance, for the deduction of a model which also include a convective effect. Here, we generalize the model in \cite{JBMS} by introducing a possibly {\em biased} movement, which leads, in general, to a convective term. As a consequence, we show the appearance of a greater variety of dynamics which allow to better investigate the long term behavior of the population; in particular, to predict its survival or extinction.

\smallskip

Our model is described by a reaction-diffusion-convection equation
\begin{equation}\label{e:E}
u_t + f(u)_x=\left(D(u)u_x\right)_x + g(u), \qquad t\ge 0, \, x\in \R,
\end{equation}
where the functions $f, D$ and $g$ satisfy \eqref{e:coeffs f}, \eqref{e:coeffs D} and \eqref{e:coeffs g},  respectively. The unknown function $u$ denotes the density (or concentration) of the population and then it has bounded range; for simplicity we assume $u\in[0,1]$.
An interesting feature of equation \eqref{e:E} in this context is that {\em negative diffusivities} arise for several natural choices of the parameters. As in \cite{JBMS}, here we consider a diffusion term which makes equation \eqref{e:E} of forward-backward-forward type. This occurrence was already noticed in other papers, see for instance \cite{SanchezGarduno-Maini-Velazquez, Turchin} in the case of a homogeneous population under different assumptions. Notice however that the deduction of the model both in \cite{JBMS} and in the present paper also involves the reaction term, while in \cite{SanchezGarduno-Maini-Velazquez, Turchin} it is limited to diffusion.
As opposite to {\em positive} diffusivities, which model the spatial spreading, {\em negative} diffusivities are usually interpreted to model the \lq\lq chaotic\rq\rq\ movement which follows from aggregation \cite{SanchezGarduno-Maini-Velazquez, Turchin}. In turn, the latter is \lq\lq a macroscopic
effect of the isolated and the grouped motility of the agents, together
with competition for space\rq\rq \cite{Li2021}. At last, we assume that the reaction term $g$ shows the strong Allee effect, i.e., it is of the so called \emph{bistable} type (see assumption (g) below).

\smallskip

We focus on the existence of  traveling-wave solutions $u(x,t)=\phi(x-ct)$ to equation \eqref{e:E}, for some profiles $\phi=\phi(\xi)$ and wave speeds $c$, see \cite{GK} for general information. If the profile is defined in $\R$, it is monotone, nonconstant, and reaches asymptotically the equilibria of \eqref{e:E}, then the corresponding traveling-wave solution is called a {\em wavefront}. We consider precisely  decreasing profiles which  connect the outer equilibria of $g$, i.e.,
\begin{equation}\label{e:bvs}
\phi(-\infty)=1 \quad \hbox{ and }\quad \phi(\infty)=0.
\end{equation}
The case when profiles are increasing, and then satisfy $\phi(-\infty)=0$, $\phi(\infty)=1$, is dealt analogously and leads to a similar discussion. These solutions, even if of a special kind, have several advantages: they are global, they are often in good agreement with experimental data \cite{Murray}, and can be attractors for more general solutions \cite{Fife}. Moreover, when $u$ represents the density of a biological species, as in this case, then condition \eqref{e:bvs} means that, for times $t\to\infty$, the species either successfully persists if $c>0$, or it becomes extinct if $c<0$. The wavefront profile $\phi$ must satisfy the ordinary differential equation
\begin{equation}\label{e:ODE}
\left(D(\phi)\phi^{\prime}\right)^{\prime}+\left(c -\dot f(\phi)\right)\phi' + g(\phi)=0.
\end{equation}
We used the notation $\dot{}:=d/du$ and ${}':=d/d\xi$. Although one can consider the case of discontinuous profiles, see \cite{LHMS, Li2021} and references there, in this paper we focus on regular monotone profiles of equation \eqref{e:ODE}. This means that they are continuous, and of class $C^2$ except possibly at points where $D$ vanishes; then solutions to equation \eqref{e:ODE} are intended in the distribution sense.

\noindent The existence of wavefronts is treated here  in a quite general framework, which includes, in particular, our biological model.
More precisely, we fix three real numbers $\alpha,\beta,\gamma$ satisfying
\begin{equation}\label{eq:condition} 0< \alpha < \gamma < \beta <1,
\end{equation}
and assume, see Figure \ref{f:Dg},
\begin{itemize}
\item[(f)] \, $f\in C^1[0,1]$;

\item[(D)] \, $D\in C^1[0,1]$, $D>0$ in $[0,\alpha) \cup (\beta,1]$, and $D<0$ in $(\alpha, \beta)$;

\item[{(g)}]\, $g\in C^1[0,1]$, $g<0$ in $(0,\gamma)$, $g>0$ in $(\gamma,1)$, and $g(0)=g(\gamma)=g(1)=0$.
\end{itemize}
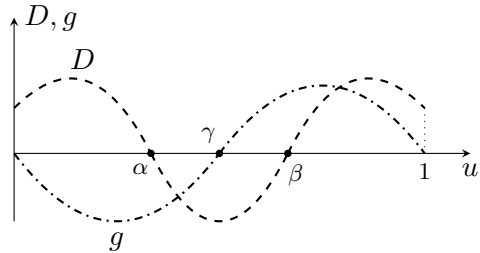
\begin{figure}[htb]
\begin{center}
\begin{tikzpicture}[>=stealth, scale=0.6]
\draw[->] (0,0) --  (10,0) node[below]{$u$} coordinate (x axis);
\draw[->] (0,0) -- (0,3) node[right]{$D,g$} coordinate (y axis);
\draw (0,0) -- (0,-1.5);
\draw[thick,dashed] (0,1) .. controls (1,2) and (2,2) .. (3,0) node[left=4, below=0]{{\footnotesize{$\alpha$}}} node[midway,above]{$D$};
\draw[thick, dashed] (3,0) .. controls (4,-2) and (5,-2) .. (6,0);
\draw[thick, dashed] (6,0) node[right=3, below=0]{{\footnotesize{$\beta$}}} .. controls (7,2) and (8,2) .. (9,1);
\draw[dotted] (9,1) -- (9,0) node[below]{{\footnotesize{$1$}}};
\draw[thick,dashdotted] (0,0) .. controls (1.5,-2) and (3,-2) .. (4.5,0) node[left=4, above=0]{{\footnotesize{$\gamma$}}} node[midway,below]{$g$};
\draw[thick,dashdotted] (4.5,0) .. controls (6,2) and (7.5,2) .. (9,0);
\filldraw[black] (3,0) circle (2pt);
\filldraw[black] (4.5,0) circle (2pt);
\filldraw[black] (6,0) circle (2pt);
\end{tikzpicture}

\end{center}
\caption{\label{f:Dg}{Typical plots of the functions $D$ (dashed line) and $g$ (dashdotted line).}}
\end{figure}

Since $f$ in \eqref{e:E} is defined up to an additive constant, we can take $f(0) =0$. The term $\dot f(u)$ represents the drift of the total concentration $u$ and prescribes in particular if a concentration wave is moving toward the right ($\dot f(u)>0$) or toward the left ($\dot f(u)<0$).
The parabolic equation \eqref{e:E} is of {\em backward} type in the interval $(\alpha,\beta)$ and of {\em forward} type elsewhere; moreover, it degenerates at $\alpha$ and $\beta$.

\smallskip

The presence of wavefronts to \eqref{e:E} satisfying (D) and (g) and with $f=0$ was first discussed in \cite{Kuzmin-Ruggerini}, where it is shown in particular that, if a wavefront exists,  then $\gamma\notin[\alpha,\beta]$. Such a situation and many others, again with $f=0$, was also considered  in \cite[cases 6.3, 8.3]{JBMS}, in the framework of the particular model deduced in that paper. The case with convection is not yet completely understood. Then our issue here is
\[
\hbox{{\em whether and when the presence of the convective flow allows the existence of wavefronts}.}
\]
An intuitive argument, see Remark \ref{r:exw-fconcave}, shows that the answer is in the affirmative at least for suitable concave $f$. We now briefly report on the content of this paper.

\smallskip

In Section \ref{s:model} we introduce the biological model and state our main results about it, for more immediacy, in a somewhat simplified way; proofs, which require the analysis of the general case dealt in Section \ref{s:tools}, and more details, are deferred to Section \ref{s:app}. In particular, we provide positive or negative results for each of the possible behaviors of $f$, namely, when $f$ is concave, convex, or it changes concavity once.

\smallskip

In Section \ref{s:tools} we investigate the  {\em fine properties} (uniqueness, strict monotonicity, estimates of speed thresholds) of such wavefronts for equation \eqref{e:E} satisfying (f)-(D) and (g).  A similar discussion for a \emph{monostable} reaction term $g$  appeared in \cite{BCM2} and \cite{BCM4}, respectively in the general framework and  for the population model with biased movements. We recall that $g$ is called {\em monostable} if $g>0$ in $(0,1)$ and $g(0)=g(1)=0$.


\smallskip As in our aforementioned papers, we exploit here an order-reduction technique. Since we focus on profiles $\phi=\phi(\xi)$ that are strictly monotone when $\phi\in(0,1)$, we can consider the inverse function $\phi^{-1}(\phi)$ of $\phi$ and, by denoting $z(\phi) := D(\phi)\phi\rq{}\left(\phi^{-1}(\phi)\right)$, we reduce the problem \eqref{e:ODE} to a first-order singular boundary-value problem for $z$ in $[0,1]$. This problem is tackled by the classical techniques of upper- and lower-solutions. This technique requires lighter assumptions than the phase-plane analysis in \cite{JBMS} and is simpler than the geometric singular perturbation theory exploited in \cite{LHMS}. Then wavefronts satisfying \eqref{e:bvs} are obtained by suitably pasting traveling waves. The results appear in Section \ref{s:tools}, they are given for an arbitrary equation \eqref{e:E} satisfying conditions (f), (D), (g), and they are original. About (g), the mere requirement that $g$ is continuous and the product $Dg$ differentiable at $0$ would be sufficient for us. Both for (D) and (g), we made slightly stricter assumptions than necessary both for simplicity and because they are satisfied by our biological model with biased movements. The cases when the internal zero of $g$ is before $\alpha$, i.e. $\gamma \in (0, \alpha)$, or after $\beta$, that is  $\gamma\in (\beta, 1)$, are not treated here. Equation \eqref{e:E} with $f=0$ admits wavefronts in these cases and we expect that they persist also in the presence of the convective effect $f$.
\smallskip

The issue of the linear stability of the wavefronts is certainly interesting; we claim that it could be developed as in \cite{LHMS, Sattinger}, with a similar discussion.
\section{A biological model with biased movements}\label{s:model}
In this section we first summarize a model for the movement of organisms recently presented in \cite{JBMS} for populations constituted by two groups of individuals. Then we show how a convective term can appear in the equation because of a biased movement. At last, we provide our results about wavefronts for such a model;
proofs are deferred to Section \ref{s:app}.

\subsection{The model}
The population is divided into {\em isolated} and {\em grouped} organisms. Both groups can move, reproduce and die, with possibly different rates. The organisms occupy the sites $jl$, for $j=0,\pm1,\pm2,\ldots$ and $l>0$; we denote by $c_j$ the probability of occupancy of the $j$-th site. Let $P_m^i$ and $P_m^g$ be the movement transitional probabilities for isolated and grouped individuals, respectively; we use the notation $P_m^{i,g}$ to indicate the two sets of parameters together. Analogously, the corresponding probabilities for birth and death are $P_b^{i,g}$ and  $P_d^{i,g}$.

Differently from \cite{JBMS}, we also introduce the parameters $a^i,b^i\ge0$ and $a^g,b^g\ge0$, which characterize a (linearly) biased movement for the isolated and grouped individuals. For the isolated individuals the bias is towards the left if $a^i-b^i>0$ and towards the right if $a^i-b^i<0$; for the grouped individuals the same occurs when either $a^g-b^g>0$ or $a^g-b^g>0$, respectively.
In the case of \cite{JBMS} one has $a^i=b^i=a^g=b^g=1$ and then $a^{i,g}-b^{i,g}=0$.

\begin{figure}[htb]
\begin{center}
\begin{tikzpicture}[>=stealth, scale=0.7]
\draw[->] (0,0) --  (16,0) coordinate (x axis);

\draw(2,-0.2) node[below]{\footnotesize$j-2$}--(2,0.2);
\draw(5,-0.2) node[below]{\footnotesize$j-1$}--(5,0.2);
\draw(8,-0.2) node[below]{\footnotesize$j$}--(8,0.2);
\draw(11,-0.2) node[below]{\footnotesize$j+1$}--(11,0.2);
\draw(14,-0.2) node[below]{\footnotesize$j+2$}--(14,0.2);

\draw[->](7.9,0.5) .. controls (7,1.5) and (6,1.5) .. (5.1,0.5) node[above, midway]{\footnotesize$a^{i,g}$};
\draw[->](8.1,0.5) .. controls (9,1.5) and (10,1.5) .. (10.9,0.5) node[above, midway]{\footnotesize$b^{i,g}$};
\end{tikzpicture}
\end{center}
\caption{\label{f:ab}{Sketch of the meaning of the parameters $a^{i,g}$ and $b^{i,g}$.}}
\end{figure}
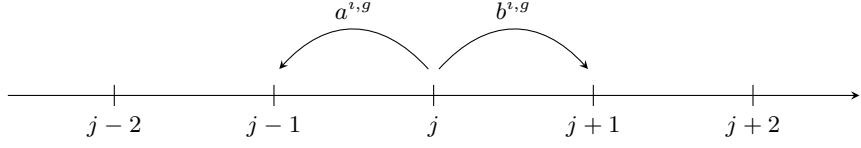

Then, the variation $\delta c_j$ of $c_j$ during a time-step $\tau>0$ is given by
\begin{align}
\lefteqn{\delta c_j=}
\nonumber
\\
= &\frac{ P_m^i}{2}\left[a^ic_{j-1}(1-c_j)(1-c_{j-2}) + b^i c_{j+1}(1-c_j)(1-c_{j+2})\right.
\left.-(a^i+b^i)c_j(1-c_{j-1})(1-c_{j+1})\right]
\nonumber
\\
&+\frac{P_m^g}{2}\left[a^g c_{j-1}(1-c_j) + b^g c_{j+1}(1-c_j) - a^g c_j(1-c_{j+1})  - b^g c_j(1-c_{j-1}) \right]
\nonumber
\\
&-\frac{P_m^g}{2}\left[a^g c_{j-1}(1-c_j)(1-c_{j-2}) + b^g c_{j+1}(1-c_j)(1-c_{j+2}) \right.
\nonumber
\\
&\quad
\left.- (a^g+b^g)c_j(1-c_{j-1})(1-c_{j+1})\right]
+\hbox{ reaction terms,}
\label{e:dtg}
\end{align}
where
\begin{align*}
\lefteqn{\hbox{ reaction terms } =}
\\
=&\frac{P_b^i}{2}\left[c_{j-1}(1-c_j)(1-c_{j-2}) + c_{j+1}(1-c_j)(1-c_{j+2})\right]
+\frac{P_b^g}{2}\left[c_{j-1}(1-c_j) + c_{j+1}(1-c_j)\right]
\\
&-\frac{P_b^g}{2}\left[c_{j-1}(1-c_j)(1-c_{j-2}) + c_{j+1}(1-c_j)(1-c_{j+2})\right]
-\frac{P_d^i}{2}\left[c_j(1-c_{j-1})(1-c_{j+1})\right]
\\
&-\frac{P_d^g}{2}\left[c_j\right] + \frac{P_d^i}{2}\left[c_j(1-c_{j-1})(1-c_{j+1})\right].
\end{align*}
By noticing that every bracket is divided by $2$, we deduce
\begin{equation}\label{e:ab-2}
a^i+b^i=a^g+b^g = 2,
\end{equation}
since in the deduction of \eqref{e:dtg} a bias $a^{i,g}$ implies a converse bias $b^{i,g}=2-a^{i,g}$.

The continuum model is obtained by replacing $c_j$ with a smooth function $c=c(x,t)$ and expanding $c$ around $x=jl$ at second order. Then, we divide \eqref{e:dtg} by $\tau$ and pass to the limit for $l,\tau\to0$ while keeping $l^2/\tau$ constant; for simplicity we assume $l^2/\tau=1$. To perform this step one makes the following assumptions on the reactive-diffusive terms \cite{JBMS}:
\begin{equation}\label{e:proba}
\frac{P_m^{i,g}l^2}{2\tau}\sim D_{i,g},\quad
\frac{P_b^{i,g}}{2\tau}\sim \lambda_{i,g},\quad
\frac{P_d^{i,g}}{2\tau}\sim k_{i,g},\quad
P_b^{i,g}, P_d^{i,g} = O(\tau),\quad \hbox{ for } l\to0,\tau\to0.
\end{equation}
The above limits define the diffusivity parameters $D_{i,g}$, the birth rates $\lambda_{i,g}$, and the death rates $k_{i,g}$; all these parameters are non-negative.
About the convection terms, we require
\begin{equation}\label{e:ab-limits}
a^{i,g}(\tau) \sim 1, \ b^{i,g}(\tau) \sim 1 \quad \hbox{ and }\quad a^{i,g}(\tau)-b^{i,g}(\tau)\sim C_{i,g}\sqrt{\tau}\quad \hbox{ for } \tau\to0,
\end{equation}
for some $C_{i,g}\in\mathbb{R}$. We stress that the parameters $C_{i,g}$ can be either positive or negative according to the values of the bias coefficients $a^{i,g}$ and $b^{i,g}$; in particular, if $C_i>0$ then we have a bias toward the left of the isolated individuals, and toward the right if $C_i<0$; the analogous bias for the grouped individuals corresponds to either $C_g>0$ (left) or $C_g<0$ (right). If $C_{i,g}=0$, then the corresponding bias is too weak to pass to equation \eqref{e:model}; with a slight abuse of terminology we say that the corresponding group has no convective movement. At last, assumption $\eqref{e:ab-limits}_1$ is compatible with \eqref{e:ab-2}; assumption $\eqref{e:ab-limits}_2$ is analogous to $\eqref{e:proba}_4$.

In conclusion, we obtain the equation
\begin{equation}\label{e:model}
u_t + f(u)_x=\left(D(u)u_x\right)_x + g(u),
\end{equation}
with
\begin{align}
\label{e:coeffs f}
f(u) &= -\left(C_iD_i+C_gD_g\right)u(1-u)^2 - C_gD_g u(1-u),
\\
\label{e:coeffs D}
D(u) &= D_i\left(1-4u+3u^2\right) + D_g\left(4u-3u^2\right),
\\
\label{e:coeffs g}
g(u) &= \lambda_g u (1-u) +\left[\lambda_i-\lambda_g- \left(k_i - k_g\right) \right] u(1-u)^2 -k_g u.
\end{align}
The model \eqref{e:model} depends on the eight parameters $C_{i,g}$, $D_{i,g}$, $k_{i,g}$ and $\lambda_{i,g}$. Equation \eqref{e:model} coincides with \eqref{e:E} but we agree that when we refer to \eqref{e:model} we understand $f$, $D$ and $g$ as in \eqref{e:coeffs f}--\eqref{e:coeffs g}. We point out that $f(0)=f(1)=0$, i.e., the convective flow vanishes when the density is either zero or maximum, as physically it should be.
When  $C_i=0$ the isolated individuals have no convective movement and the function $f$ is convex in $[0,1]$ if $C_g>0$ and concave otherwise. Instead, when $C_g=0$ the grouped individuals have no convective movements and $f$ changes its concavity for $u=2/3$. The diffusion and reaction terms \eqref{e:coeffs D}-\eqref{e:coeffs g} coincide with those in \cite[(2)]{JBMS}, while $f$ is missing there.



\subsection{Main results on the model}
About the model introduced in the previous subsection, the case we are interested in is  when  conditions (D) and (g) are satisfied; the corresponding assumptions on the parameters have already been given in \cite{JBMS, Li2021}.

\begin{lemma}\label{l:la}
The diffusivity $D$ in \eqref{e:coeffs D} satisfies {\rm (D)} if and only if $D_i>4D_g>0$. In this case we have
\begin{equation}\label{e:alphabeta}
\alpha=\frac23 - \frac\omega3 \quad \hbox{ and }\quad \beta=\frac23 + \frac\omega3, \qquad  \hbox{ for } \omega :=\sqrt{\frac{D_i-4D_g}{D_i-D_g}}.
\end{equation}
The reaction term $g$ in \eqref{e:coeffs g} satisfies {\rm (g)} if and only if $k_g=0, \, \lambda_g>0$ and $r_i:=k_i-\lambda_i>0$. In this case
\begin{equation}\label{e:gamma}
\gamma=\frac{r_i}{r_i+\lambda_g},
\end{equation}
and $\gamma\in(\alpha,\beta)$ if and only if
\begin{equation*}
\frac{1-\omega}{2+\omega} < \frac{\lambda_g}{r_i}<\frac{1+\omega}{2-\omega}.
\end{equation*}
\end{lemma}

Notice that $\omega \in (0,1)$, $\beta-\alpha =2\omega/3$ and $\alpha+\beta = 4/3$. The condition $k_g=0$ clearly has no biological sense, and is interpreted in the sense that the life expectancy of grouped individuals is much larger than that of isolated individuals; the condition $r_i>0$ further hinders the latter. Here, $\gamma$ is the Allee parameter \cite{JBMS, Li2021}.

Here follows a simple necessary condition for the existence of wavefronts.

\begin{proposition}\label{p:CN}
If \eqref{e:model} admits wavefronts satisfying condition \eqref{e:bvs}, then $C_g<0$.
\end{proposition}
It is easy to see that a necessary condition to have wavefronts satisfying conditions $\phi(-\infty)=0$ and $\phi(\infty)=1$, instead of \eqref{e:bvs}, is $C_g>0$.

We now summarize the restrictions required on the parameters:
\begin{gather}
C_g<0,\qquad
D_i>4D_g>0,
\label{e:restr1-1}
\\
r_i:= k_i-\lambda_i>0, \qquad k_g=0, \qquad \lambda_g>0,
\label{e:restr1-2}
\\
\ds\frac{1-\omega}{2+\omega} < \frac{\lambda_g}{r_i}<\frac{1+\omega}{2-\omega},
\label{e:restr1-3}
\end{gather}
with $\omega$ defined in \eqref{e:alphabeta}. We {\em always} assume conditions \eqref{e:restr1-1}--\eqref{e:restr1-3} in the following, without any further mention. The results below are preferably stated by referring to the following dimensionless quotients and by lumping the parameters referring to the grouped population into a single dimensionless parameter as follows:
\begin{equation}\label{e:def-cd}
\c:=\frac{C_i}{\vert C_g \vert},
\qquad
d:=\frac{D_i}{D_g}>4,
\qquad 	
\mu := \frac{r_i}{\lambda_g}= \frac{k_i-\lambda_i}{\lambda_g},
\qquad
E_g:=|C_g|\sqrt{\frac{D_g}{\lambda_g}}.
\end{equation}
Under this notation we have
\begin{equation}\label{e:omegadad}
\omega = \sqrt{\frac{d-4}{d-1}}.
\end{equation}
Notice that $E_g$ gathers the parameters concerning convection, diffusion and reaction of the grouped individuals; the parameter $\mu$ is the ratio between the net increasing rate of the isolated and grouped individuals. Notice that condition \eqref{e:restr1-3} is equivalent to
\begin{equation}\label{e:restr1-3-mu}
\frac{2-\omega}{1+\omega} < \mu <\frac{2+\omega}{1-\omega}.
\end{equation}

The convective term $f$ can change convexity at most once; then, it can be either concave or convex, or else convex-concave or concave-convex. We now examine each of these cases; in all of them, we emphasize that $s$ is always multiplied by $d$. Since the parameter $s$ does not depend on $d$, we can understand $sd$ as a variable  independent from $d$, which lumps the ratios of the coefficients related to the movement. In this way we shall often deal with the couple $(\omega, sd)$ of parameters, where $\omega$ depends on $d$.

\paragraph{The concave case} The convective term $f$ is strictly concave if and only if
\begin{equation}\label{e:restr2}
0 \le \c d \le\frac{3}{2},
\end{equation}
(see Lemma \ref{lem:conc}). In this case, model \eqref{e:model} admits wavefronts satisfying \eqref{e:bvs} and we can also discuss  the sign of their speeds $c$; we now provide some prototype results. The key condition is \begin{equation}\label{e:argen}
r_i(2+\omega)+\lambda_g(1+\omega) < \frac{8 C_g^2D_g}{81} \frac{d-4}{(d-1)^2}.
\end{equation}
Condition \eqref{e:argen} contains several parameters; therefore, there are many ways of discussing the results, depending on which parameters are set and which are held constant; we focus on two different choices.

First, for fixed $C_g, D_g$ we consider the triangle, see Figure \ref{f:triangle} on the left,
\begin{equation}\label{e:triangleT}
\mathcal{T}_g(d):=\left\{(r_i,\lambda_g)\in\R^+\times\R^+
\colon
\eqref{e:restr1-3} \hbox{ and }\eqref{e:argen} \hbox{ hold }\right\}.
\end{equation}
Under \eqref{e:restr2} if $(r_i,\lambda_g)\in \mathcal{T}_g(d)$ then  equation \eqref{e:model} admits wavefronts satisfying \eqref{e:bvs} (see Remark \ref{rem:specific1}.

\begin{figure}[htbp]
    \centering
    \begin{tabular}{cc}
    \includegraphics[width=7.2cm]{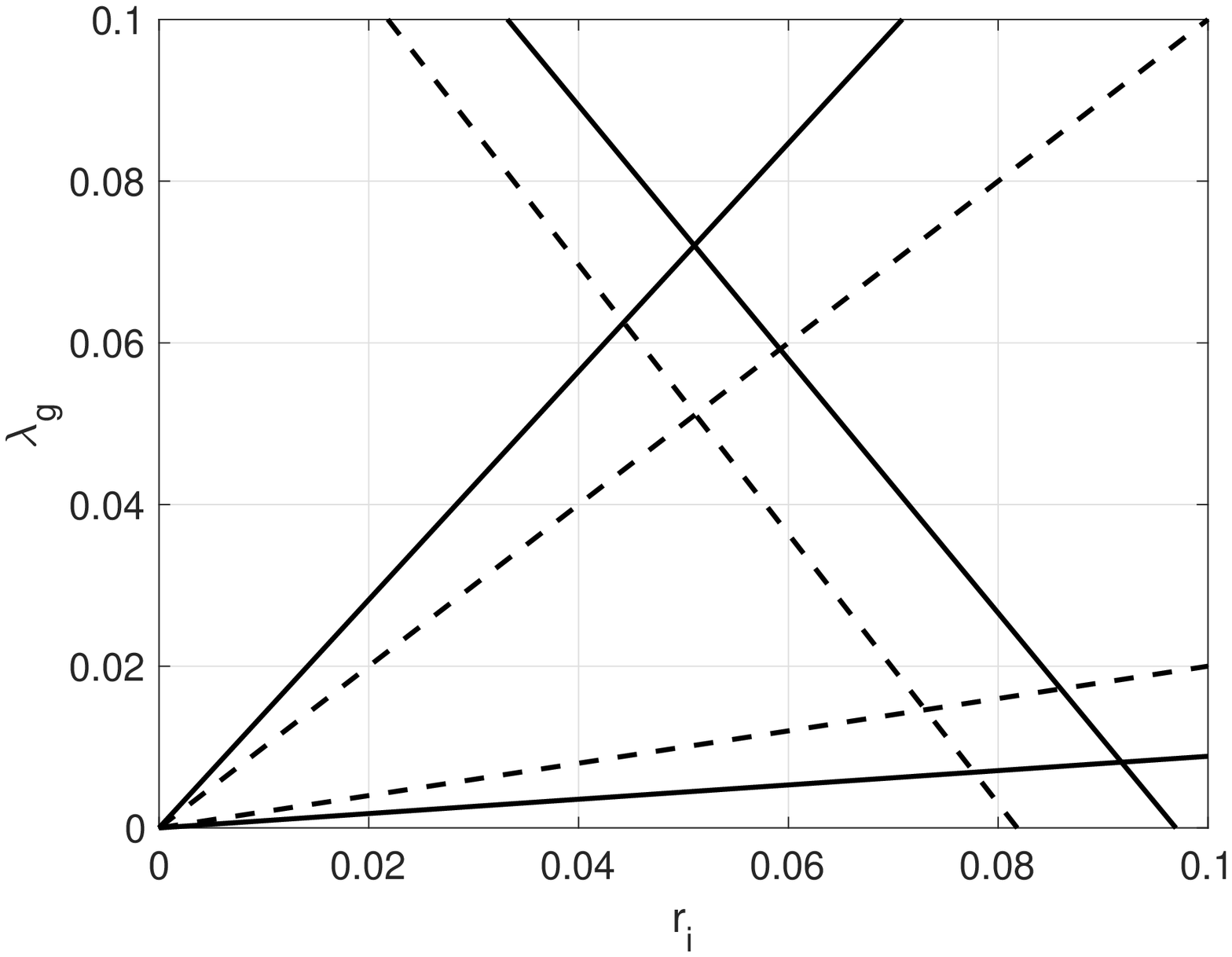}
    &
    \includegraphics[width=7.2cm]{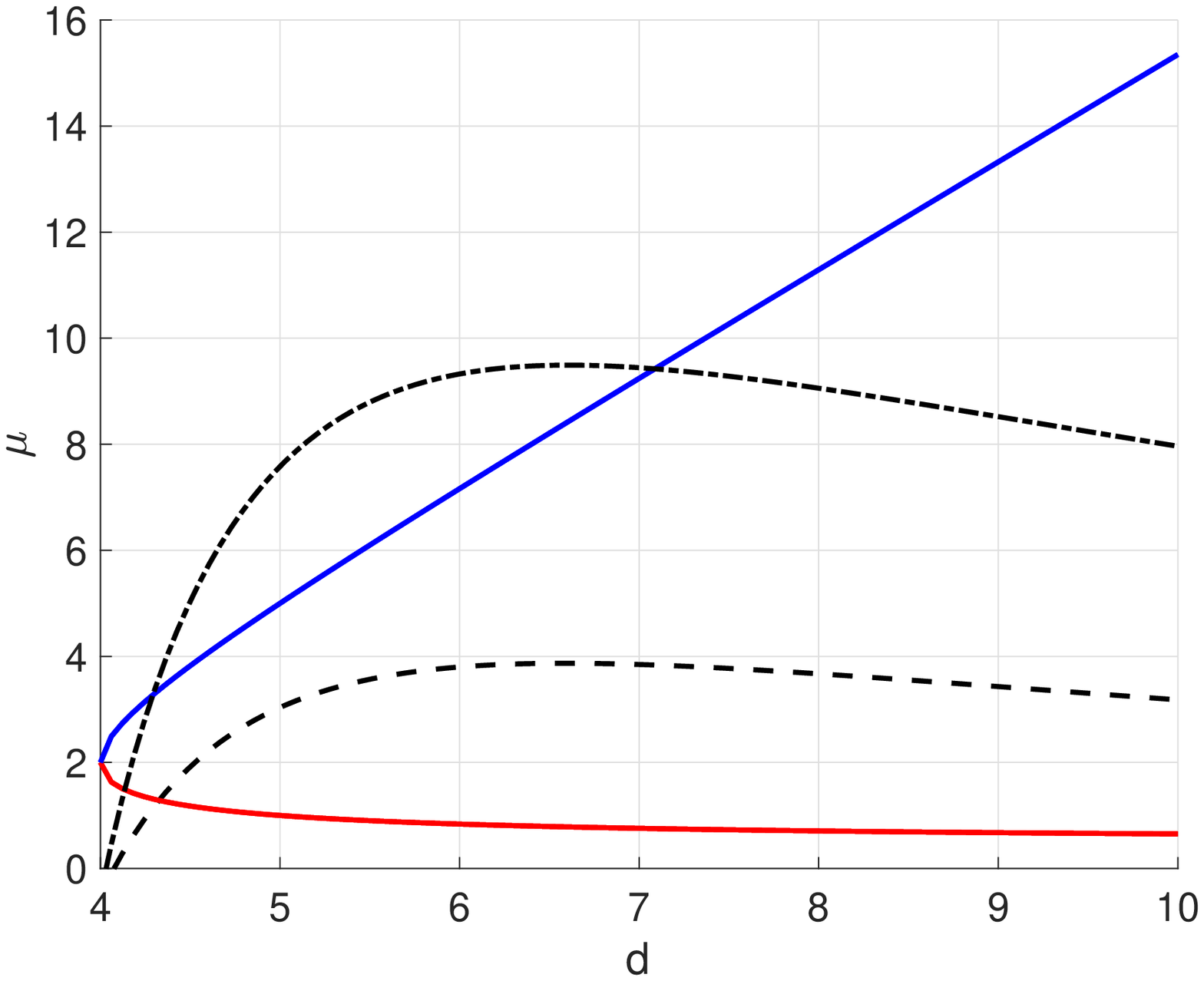}
    \end{tabular}
  \caption{On the left: the triangle $\mathcal{T}_g(d)$ is the intersection of three half-planes. Here $|C_g|\sqrt{D_g}=6$; dashed lines refer to $d=5$, solid lines to $d= 8$. On the right: The conditions in \eqref{e:restr1-3-mu} prescribe that $\mu$ must lie between the red and the blue line. Condition \eqref{e:argen} further prescribes that $\mu>0$ must belong to the region below the black line: the dashed curve refers to $E_g=40$, the dash-dotted curve to $E_g=60$.}
    \label{f:triangle}
    \end{figure}

Moreover, for such $(r_i,\lambda_g)\in \mathcal{T}_g(d)$ we have, see Figure \ref{f:lined}:

\begin{itemize}

\item if $d>4+2\sqrt{3}$ and also
\begin{equation}\label{e:ppoo}
\frac{18\sqrt{\mu (d-1)}}{\tau(\omega,\gamma,\c d)}< E_g,
\end{equation}
where $\tau$ is defined in \eqref{e:tau}, then \eqref{e:model} admits wavefronts satisfying \eqref{e:bvs} with {\em positive} speeds (see Remark \ref{rem});

\item if $d\in(4,(5+2\sqrt{3})/2)$, then {\em every} pair $(r_i,\lambda_g)\in\mathcal{T}_g(d)$ provides profiles and all of them have {\em negative} speeds (see Remark \ref{rem:negative}).

\end{itemize}

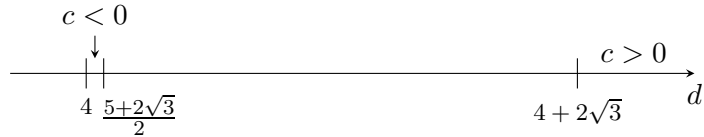
\begin{figure}[htb]
\begin{center}
\begin{tikzpicture}[>=stealth, scale=1]
\draw[->] (0,0) --  (9,0) node[below]{$d$} coordinate (x axis);
\draw(1,-0.2) node[below]{\footnotesize$4$}--(1,0.2);
\draw(1.23,-0.2) node[below=10,right=-5]{$\frac{5+2\sqrt{3}}{2}$}--(1.23,0.2);
\draw[->](1.115,0.5) node[above]{$c<0$} --(1.115,0.2);
\draw(7.46,-0.2) node[below]{\footnotesize$4+2\sqrt{3}$}--(7.46,0.2);
\draw(8.2,0) node[above]{$c>0$};
\end{tikzpicture}

\end{center}
\caption{\label{f:lined}{Thresholds leading to wavefronts with positive or negative speeds when $f$ is strictly concave. The figure is to be interpreted as follows: if $d>4+2\sqrt3\sim 7.46$, then there exist wavefronts with positive speed, while, if $4<d<(5+2\sqrt3)/2\sim 4.23$, then every wavefront has negative speed under the assumptions of Theorem \ref{t:modelc<0}.}}
\end{figure}

Second, notice that, assuming again \eqref{e:restr2}, we can interpret \eqref{e:restr1-3} and \eqref{e:argen} as relationships between $d$ and $\mu$ for fixed $E_g$, see Figure \ref{f:triangle} on the right and Corollary \ref{c:lmmk}. In this framework, profiles exist for every couple $(d,\mu)$ lying in the region between the red and blue lines, and below the black line.

\paragraph{The convex case}
If $f$ is convex, then equation \eqref{e:model} admits no such wavefronts. Indeed, we show a stronger result: wavefronts satisfying \eqref{e:bvs} never exist if $f$ is  convex just in the interval $[\alpha,\beta]$ (see Remark \ref{rem:convf}).

\paragraph{The case when $f$ changes convexity} We now consider the two cases when $f$ changes convexity; in order to simplify the analysis, we focus on the particular case when the inflection point of $f$ coincides with $\gamma$. Recall that $\gamma$ represents the Allee parameter \cite{JBMS}, which describes the threshold separating a decrease of concentration (if $u<\gamma$) from an increase of concentration (if $u>\gamma$). The assumption that $f$ has an inflection point at $\gamma$ means that the maximum drift $\dot f$ (if $f$ is convex-concave) or the minimum drift (if $f$ is concave-convex) is precisely reached at $\gamma$. We refer to Figure \ref{f:twof} for an illustration of both cases.

\begin{figure}[htb]
\begin{center}
\begin{tikzpicture}[>=stealth, scale=0.6]
\draw[->] (0,0) --  (10,0) node[below]{$u$} coordinate (x axis);
\draw[->] (0,0) -- (0,2.5) coordinate (y axis);
\draw (0,0) -- (0,-1.5);
\draw[thick,dashed] (0,1) .. controls (1,2) and (2,2) .. (3,0) node[left=4, below=0]{{\footnotesize{$\alpha$}}} node[midway,above]{$D$};
\draw[thick,dashed] (3,0) .. controls (4,-2) and (5,-2) .. (6,0);
\draw[thick,dashed] (6,0) node[right=3, below=0]{{\footnotesize{$\beta$}}} .. controls (7,2) and (8,2) .. (9,1);
\draw[dotted] (9,1) -- (9,0) node[below]{{\footnotesize{$1$}}};
\draw[thick,dashdotted] (0,0) .. controls (1.5,-2) and (3,-2) .. (4.5,0) node[left=4, above=0]{{\footnotesize{$\gamma$}}} node[midway,below]{$g$};
\draw[thick,dashdotted] (4.5,0) .. controls (6,2) and (7.5,2) .. (9,0);
\filldraw[black] (3,0) circle (2pt);
\filldraw[black] (4.5,0) circle (2pt);
\filldraw[black] (6,0) circle (2pt);
\draw[thick] (0,0) .. controls (1,-3) and (4,0) .. (4.5,1);
\draw[thick] (4.5,1) .. controls (5,2) and (7,3) .. (9,0) node[midway,above]{$f$};
\draw[dotted](4.5,0)--(4.5,1);

\begin{scope}[xshift=12cm]
\draw[->] (0,0) --  (10,0) node[below]{$u$} coordinate (x axis);
\draw[->] (0,0) -- (0,2.5) coordinate (y axis);
\draw (0,0) -- (0,-1.5);
\draw[thick,dashed] (0,1) .. controls (1,2) and (2,2) .. (3,0) node[left=4, below=0]{{\footnotesize{$\alpha$}}} node[near end,above]{$D$};
\draw[thick,dashed] (3,0) .. controls (4,-2) and (5,-2) .. (6,0);
\draw[thick,dashed] (6,0) node[right=3, below=0]{{\footnotesize{$\beta$}}} .. controls (7,2) and (8,2) .. (9,1);
\draw[dotted] (9,1) -- (9,0) node[below]{{\footnotesize{$1$}}};
\draw[thick,dashdotted] (0,0) .. controls (1.5,-2) and (3,-2) .. (4.5,0) node[left=10, above=0]{{\footnotesize{$\gamma$}}} node[midway,above]{$g$};
\draw[thick,dashdotted] (4.5,0) .. controls (6,2) and (7.5,2) .. (9,0);
\filldraw[black] (3,0) circle (2pt);
\filldraw[black] (4.5,0) circle (2pt);
\filldraw[black] (6,0) circle (2pt);
\draw[thick] (0,0) .. controls (1,3) and (3,3) .. (4.5,1) node[near end,above]{$f$};
\draw[thick] (4.5,1) .. controls (6,-1) and (7.5,-3) .. (9,0) node[midway,above]{$f$};
\draw[dotted] (4.5,0)--(4.5,1);

\end{scope}
\end{tikzpicture}

\end{center}
\caption{\label{f:twof}{Plots of the functions $D$ (dashed line), $g$ (dashdotted line) and $f$ (solid line) in the case $f$ is convex-concave (on the left) and concave-convex (on the right), with $\gamma$ as inflection point of $f$.}}
\end{figure}
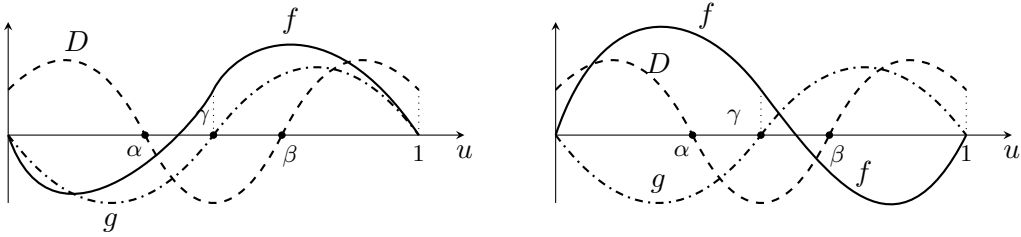

\begin{itemize}

\item Assume $f$ is convex-concave and define the set
\begin{equation}\label{e:setS}
\mathcal{S}:=\left\{(\omega, \c d) \colon  1+\frac{1}{\omega}<\c d <\frac{12(2+3\omega)}{(4+\omega)^2}\right\}.
\end{equation}
If $(\omega, sd)\in\mathcal{S}$ and estimate \eqref{e:formula} holds, then equation \eqref{e:model} has wavefronts satisfying condition \eqref{e:bvs}.

\item Assume $f$ is concave-convex and define the set
\begin{equation}\label{e:setStilde}
\tilde{\mathcal{S}}:=\left\{(\omega, \c d) \colon -\frac{16\omega}{(\omega +2)^2} <\c d<1-\frac{1}{\omega} \right\}.
\end{equation}
If $(\omega, sd)\in\tilde{\mathcal{S}}$ and estimate \eqref{e:formula2.2} holds, then equation \eqref{e:model} has wavefronts satisfying condition \eqref{e:bvs}.

\end{itemize}

\noindent We refer to Figure \ref{f:StildeS} for an illustration of the sets $\mathcal{S}$ and $\tilde{\mathcal{S}}$.
%
\begin{figure}[htbp]
    \centering
    \begin{tabular}{cc}
    \includegraphics[width=7.2cm]{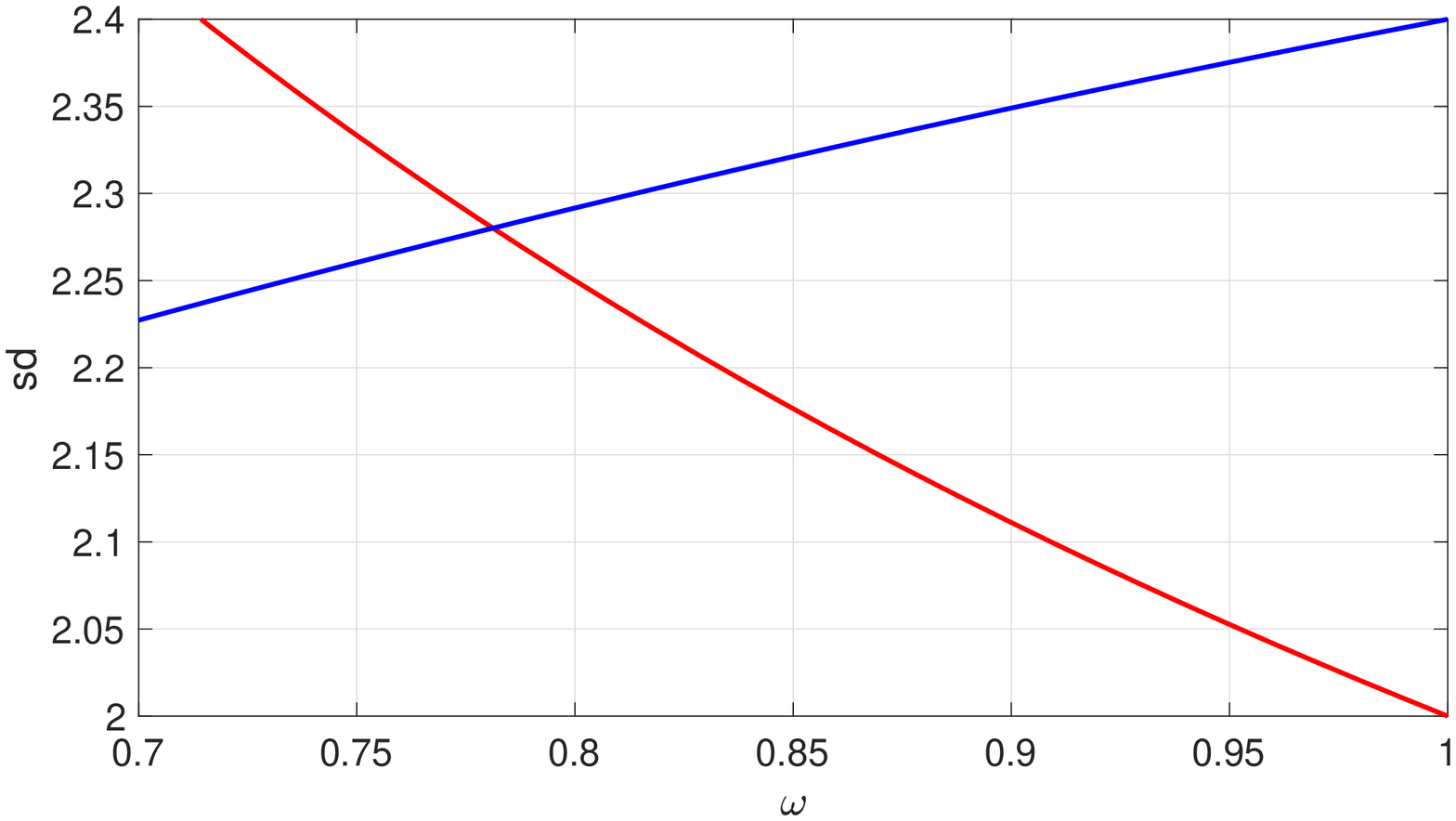}
    &
    \includegraphics[width=7.2cm]{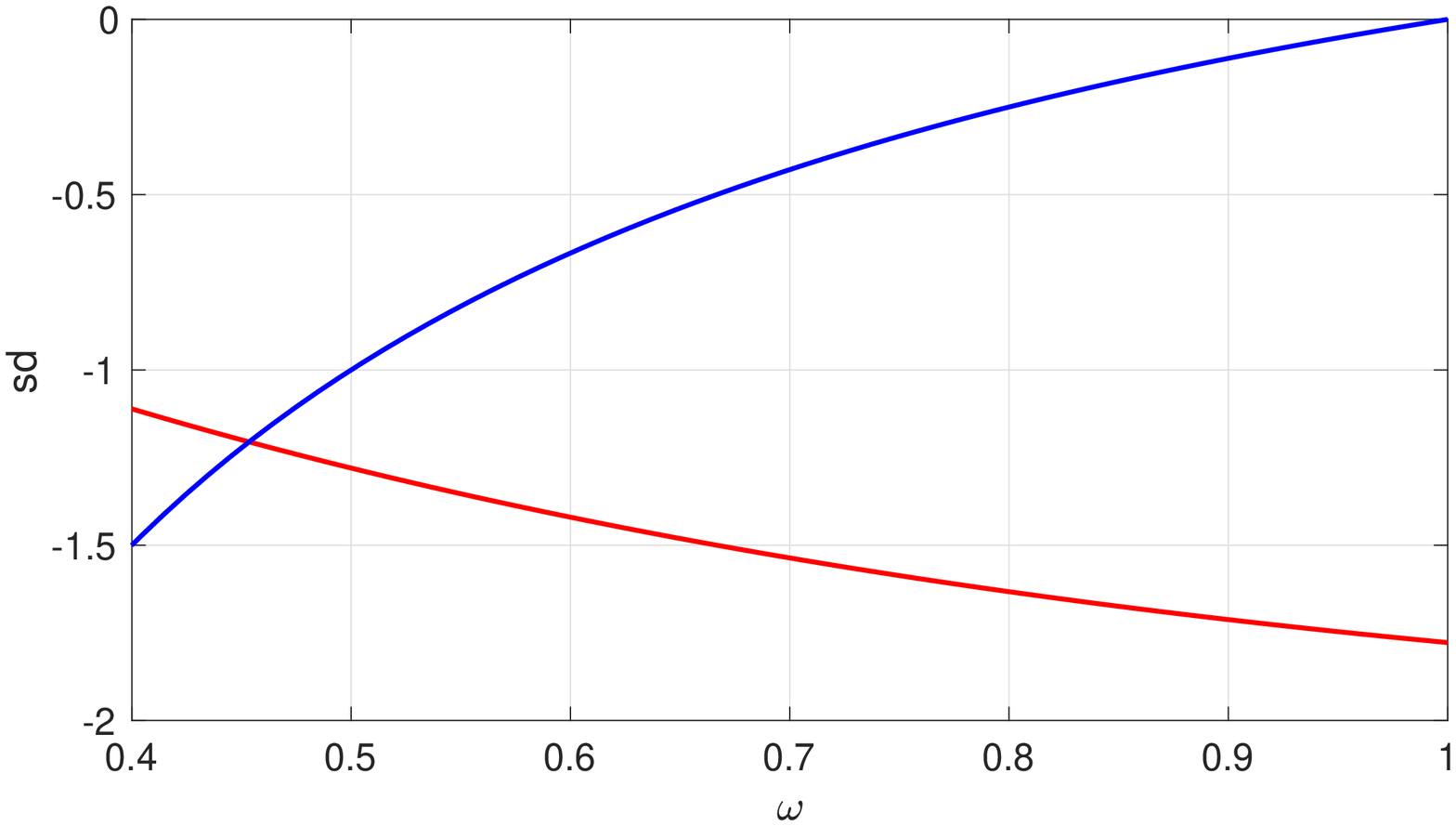}
    \end{tabular}
  \caption{The sets $\mathcal{S}$, on the left, and $\tilde{\mathcal{S}}$, on the right. The sets $\mathcal{S}$ and $\tilde{\mathcal{S}}$ are the plane regions bounded from above by the blu curve and from below by the red curve.}
    \label{f:StildeS}
    \end{figure}


\section{Theoretical results}\label{s:tools}
\setcounter{equation}{0}

In this section we provide the theoretical results that are needed for the investigation of model \eqref{e:model}.
In the following we consider equation \eqref{e:E} and we always assume \eqref{eq:condition} and {\rm (f)}, {\rm (D)}, {\rm (g)}, without any further mention.
The existence of a wavefront solution to \eqref{e:E}, whose profile satisfies \eqref{e:bvs}, is obtained by pasting profiles connecting $0$ with $\alpha$, $\alpha$ with $\gamma$, $\gamma$ with $\beta$, and $\beta$ with $1$. Each subprofile exists for $c$ larger or smaller than a certain threshold, which varies according to the subinterval: we denote them by
\[
c_{0,\alpha}^*,\quad c_{\alpha,\gamma}^*,\quad c_{\gamma,\beta}^*,\quad c_{\beta,1}^*,
\]
respectively. The expressions of these thresholds are not explicit, but we provide below rather precise estimates for them. We denote
\begin{equation}\label{e:stima-def}
c_0:=\min\{c^*_{0, \alpha}, \, c^*_{\alpha, \gamma}\} \quad \hbox{ and }\quad c_1=:\max\{c^*_{\gamma, \beta}, \, c^*_{\beta, 1}\}.
\end{equation}
In the above pasting framework, $c_0$ involves the speeds of profiles connecting $0$ with $\alpha$ and then $\alpha$ to $\gamma$, while $c_1$ refers to the connections $\gamma$ to $\beta$ and then $\beta$ to $1$. We denote with $\mathcal{J}$ the set of {\em admissible speeds}, i.e., the speeds $c$ such that there is a profile with that speed satisfying \eqref{e:bvs}.

The following main result concerns general necessary and sufficient conditions for the existence of wavefronts.

\begin{theorem}\label{t:solvability}
If
\begin{equation}\label{e:stima}
c_1<c_0,
\end{equation}
then for every $c \in (c_1, c_0)$ there are wavefronts to equation \eqref{e:E} satisfying \eqref{e:bvs}.

Conversely, if $c_1>c_0$ there exists no wavefronts to equation \eqref{e:E} satisfying \eqref{e:bvs}.
\end{theorem}

We point out that, in the case $c_1<c_0$, there are {\em infinitely many} profiles with the {\em same} speed $c\in(c_1,c_0)$. More precisely, see Figure \ref{f:slopes}, for every such $c$ there exists $\lambda_c<0$ and a family of profiles $\phi_\lambda$, for $\lambda\in[\lambda_c,0)$, which are characterized by
\[
\phi(\xi_\gamma)= \gamma \quad \hbox{ and }\quad \phi_\lambda'(\xi_\gamma)=\lambda,
\]
for some $\xi_\gamma\in\mathbb{R}$. The first condition simply says that all profiles have been shifted so that they reach the value $\gamma$ at the same $\xi=\xi_\gamma$ (in order to make a comparison possible); the second one states that their slopes at $\xi_\gamma$ cover the cone centered at $(\xi_\gamma,\gamma)$ and opening $[\lambda_c, 0)$. We refer to \cite{BCM3} for more information.

\begin{figure}[htb]
\begin{center}

\begin{tikzpicture}[>=stealth, scale=0.5]

\draw[->] (-7,0) --  (7,0) node[below]{$\xi$} coordinate (x axis);
\draw[->] (-6,0) -- (-6,6) node[right]{$\phi$} coordinate (y axis);
\draw (-6,5) node[left=3, above]{\footnotesize$1$} -- (7,5);
\draw (-7,5) -- (-6,5);
\draw[thick] (-7,4.5) .. controls (-4,4.4) and (-3,2.6) .. (0,2.5);
\draw[dotted] (0,0) node[below]{\footnotesize{$\xi_\gamma$}} -- (0,2.5);
\draw[dotted] (-6,2.5) node[left]{\footnotesize{$\gamma$}} -- (0,2.5);
\draw[thick] (0,2.5) .. controls (3,2.4) and (4,0.6) .. (7,0.5);

\draw[thick, dashed] (-7,4.7) .. controls (-4,4.6) and (-2,3) .. (0,2.5);
\draw[thick,dashed] (0,2.5) .. controls (2,2) and (4,0.4) .. (7,0.3);

\draw[thick, dashdotted] (-7,4.9) .. controls (-4,4.9) and (-1,3.5) .. (0,2.5);
\draw[thick,dashdotted] (0,2.5) .. controls (1,1.5) and (4,0.1) .. (7,0.1);

\draw (-3,2.5) -- (3,2.5) node[right]{\footnotesize{$\lambda=0$}};
\draw (2.5,0) -- (-2.5,5) node[right=30, below=1]{\footnotesize{$\lambda=\lambda_c$}};

\end{tikzpicture}
\end{center}
\caption{\label{f:slopes}{Some profiles with the same speed $c$, in the case $c_1<c_0$ in Theorem \ref{t:solvability}}; profiles have been shifted so that $\phi(\xi_\gamma)=\gamma$.}
\end{figure}
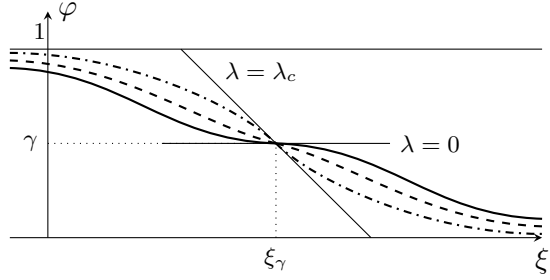

We denote the {\em difference quotient} of a scalar function of a real variable $F=F(\phi)$ with respect to a point $\phi_0$ as
\begin{equation}\label{e:dq}
\delta(F,\phi_0)(\phi) := \frac{F(\phi)-F(\phi_0)}{\phi-\phi_0}.
\end{equation}
We also introduce the integral mean of the difference quotient and denote it as
\[
\Delta(F,\phi_0)(\phi) := \frac{1}{\phi-\phi_0}\int_{\phi_0}^\phi \delta(F,\phi_0)(\psi)\, d\psi= \frac{1}{\phi-\phi_0}\int_{\phi_0}^\phi \frac{F(\psi)-F(\phi_0)}{\psi-\phi_0}\, d\psi.
\]
Notice that for every $\phi \in (0, \gamma)$ there is $\phi_1 \in (0, \gamma)$ such that
$\Delta(Dg, \alpha)(\phi)=\delta(Dg, \alpha)(\phi_1)$.
Then we have
\begin{equation}\label{e:deltaDelta}
\sup_{[0, \gamma]}\Delta(Dg, \alpha) \le \sup_{[0, \gamma]}\delta(Dg, \alpha),
\end{equation}
and the same estimate holds true in the interval $[\gamma,1]$ by replacing $\alpha$ with $\beta$.

The following results provide necessary and sufficient conditions for the existence of decreasing profiles, in order to make condition \eqref{e:stima} more explicit in terms of the functions $f$, $D$ and $g$. The proofs are deferred to the end of this section.

\begin{corollary}[Necessary condition]
\label{nec}
If there are wavefronts to equation \eqref{e:E} whose profiles satisfy \eqref{e:bvs}, then

\begin{equation}
\label{nec0}
\min\left\{\inf_{[0,\gamma]}\delta(f,\alpha), \dot{f}(\alpha)-2\sqrt{\dot{D}(\alpha)g(\alpha)}\right\} \ge \max \left\{ \sup_{[\gamma,1]}\delta(f,\beta),  \dot{f}(\beta)+2\sqrt{\dot{D}(\beta)g(\beta)}\right\}.
\end{equation}
\end{corollary}

In particular, wavefronts exist only if both conditions
\begin{align}
\label{nec1}
\dot{f}(\alpha) -\dot{f}(\beta) &\ge 2\sqrt{\dot{D}(\alpha)g(\alpha)} +  2\sqrt{\dot{D}(\beta)g(\beta)},
\\
\label{nec2}
\inf_{[0,\gamma]}\delta(f,\alpha) - \sup_{[\gamma,1]}\delta(f,\beta)&\ge0,
\end{align}
are satisfied. Notice that inequality \eqref{nec1} separates the behavior of $f$ from that of $Dg$.

\begin{remark}\label{r:lhswith f concava}
When $f$ is strictly concave we have
\begin{equation}\label{e:pozz1}
\inf_{[0, \gamma]} \delta(f, \alpha)=\delta(f, \alpha)(\gamma)=\frac{f(\gamma)-f(\alpha)}{\gamma-\alpha}, \quad \sup_{[ \gamma, 1]} \delta(f, \beta)=\delta(f, \beta)(\gamma)=\frac{f(\gamma)-f(\beta)}{\gamma-\beta},
\end{equation}
and then
\begin{equation}\label{e:pozz2}
\inf_{[0, \gamma]}\delta(f, \alpha)-\sup_{[ \gamma, 1]} \delta(f, \beta)=\frac{f(\gamma)-f(\alpha)}{\gamma-\alpha}-\frac{f(\gamma)-f(\beta)}{\gamma-\beta}>0.
\end{equation}
\end{remark}

The following result shows, in particular, how far from zero must be the difference in \eqref{nec2} in order to have solutions.

\begin{corollary}[Sufficient condition]\label{t:CS}
We have the following results.

\begin{itemize}

\item Assume

\begin{equation}\label{e:ssigma}
\inf_{[0,\gamma]}\delta\left(f,\alpha\right) -\sup_{[\gamma,1]}\delta\left(f,\beta\right)
>
2\sup_{[0,\gamma]}\sqrt{\Delta\left(Dg,\alpha\right)} + 2\sup_{[\gamma,1]}\sqrt{\Delta\left(Dg,\beta\right)}.
\end{equation}
Then, Equation \eqref{e:E} admits wavefronts satisfying \eqref{e:bvs}, and $\mathcal{J}$ is a bounded interval.

\item We have
\begin{enumerate}[{ (i)}]
\item either $\mathcal{J} \subset(0,\infty)$ or $\mathcal{J}=\emptyset$, in the case
\begin{equation}
\label{e:positive speeds}
\max\left\{\sup_{[\gamma,1]}\delta(f,\beta),   \dot{f}(\beta) + 2\sqrt{\dot{D}(\beta)g(\beta)}\right\}>0;
\end{equation}

\item
either $\mathcal{J} \subset(-\infty,0)$ or $\mathcal{J}=\emptyset$, in the case
\begin{equation}
\label{e:negative speeds}
\min\left\{\inf_{[0,\gamma]}\delta(f,\alpha),   \dot{f}(\alpha) - 2\sqrt{\dot{D}(\alpha)g(\alpha)}\right\}<0.
\end{equation}
\end{enumerate}
\end{itemize}
\end{corollary}


In the proof of Theorem \ref{t:solvability} we will reduce the existence of a wavefront to equation \eqref{e:E} satisfying  \eqref{e:bvs} to the investigation of a solution $z$ to the following {\em singular} first-order problem in the interval $[0,1]$:
\begin{equation}
\label{e:P}
\begin{cases}
\dot{z}(\phi)=\dot{f}(\phi) -c- \frac{D(\phi)g(\phi)}{z(\phi)} \ &\mbox{ in } \ (0,\alpha)\cup (\alpha,\beta)\cup (\beta, 1),\\
z<0 \ &\mbox{ in } \ (0,\alpha)\cup(\beta, 1),\\
z>0 \ &\mbox{ in } \ (\alpha,\beta),\\
z(0)=z(\alpha)=z(\beta)=z(1)=0.
\end{cases}
\end{equation}
By a solution to \eqref{e:P} we mean a function $z(\phi)$ which is continuous on $[0,1]$ and satisfies the equation \eqref{e:P}$_1$ in integral form, i.e.,
\[
z(\phi)=f(\phi)-c\phi-\int_{0}^{\phi}\frac{D(\sigma)g(\sigma)}{z(\sigma)}\, d\sigma, \qquad \phi \in [0, 1].
\]
Notice that we exploited here the assumption $f(0)=0$. It is clear that such a $z$ belongs to $C^1\left((0,1)\setminus\{\alpha, \beta \}\right)$.
To solve problem \eqref{e:P} we divide it into four subproblems, which correspond to the subintervals $[0, \alpha]$,  $[\alpha, \gamma]$,  $[\gamma, \beta]$ and $[\beta, 1]$ of the interval $[0,1]$. In order to have a unified treatment of any of these problems, we now collect results from \cite[Lemma 4.1, Corollary 4.1, Remark 4.1]{BCM3} for the problem
\begin{equation}
\label{e:problem0101}
\begin{cases}
\dot z(\phi) = h(\phi)-c- \frac{Q(\phi)}{z(\phi)}, \ &\phi \in (\sigma_1,\sigma_2),\\
z(\phi)<0, \ &\phi \in (\sigma_1, \sigma_2).
\end{cases}
\end{equation}

\begin{lemma}
\label{lem:0101}
Let $h$ and $Q$ be continuous functions on $[\sigma_1, \sigma_2]$, with $Q >0$ in $(\sigma_1, \sigma_2)$ and $Q(\sigma_1)=Q(\sigma_2)=0$. Then we have:
\begin{enumerate}[(a)]

\item For any $c \in \R$ there exists a unique $\zeta_c \in C^0[\sigma_1,\sigma_2] \cap C^1\left(\sigma_1,\sigma_2\right)$ satisfying \eqref{e:problem0101} and $\zeta_c(\sigma_2)=0$.
\item Denote $c^*(\sigma_1,\sigma_2):=\sup \left\{c\in \R: \zeta_c(\sigma_1)<0\right\}\in (-\infty,\infty]$.
If $c^*(\sigma_1,\sigma_2) < \infty$, then for every $c>c^*(\sigma_1,\sigma_2)$, there exists $\beta(c) \in (-\infty,0)$ such that there is a unique $z_{c,s} \in C^0[\sigma_1,\sigma_2]\cap C^1\left(\sigma_1,\sigma_2\right]$ satisfying \eqref{e:problem0101}, $z_{c,s}(\sigma_1)=0$, $z_{c,s}(\sigma_2)=s<0$, if and only if $s \ge \beta(c)$.
Moreover, we have
\begin{multline}\label{e:estimc12}
 \max \left\{\sup_{(\sigma_1, \sigma_2]}\delta\left(f, \sigma_1\right), h(\sigma_1)+2\sqrt{\dot{Q}(\sigma_1)}\right\} \le  c^*(\sigma_1,\sigma_2) \le
\\
  \sup_{(\sigma_1, \sigma_2]}\delta\left(f, \sigma_1\right) + 2 \sup_{ (\sigma_1, \sigma_2]}\sqrt{\Delta(Q, \sigma_1)},
\end{multline}
where $f(\phi):=\int_{0}^{\phi}h(\sigma)\, d\sigma, \, \phi \in [0,1]$.

\item If $\dot{Q}(0)$ exists, then $c^*(\sigma_1,\sigma_2)$ is finite.

\end{enumerate}
\end{lemma}

Conditions \eqref{e:estimc12} also exploit estimates on the threshold speeds recently proposed in  \cite{Marcelli-Papalini}. With the help of Lemma \ref{lem:0101},  in the proof of the following proposition we analyze the subproblems we mentioned above.

\begin{proposition}\label{prop:solvability}
Problem \eqref{e:P} is solvable if $c_1 < c_0$ and it is not solvable if $c_1>c_0$; in the former case we have $c\in [c_1, c_0]$.

Estimates for the thresholds $c_{0,\alpha}^*$, $c_{\alpha,\gamma}^*$, $c_{\gamma,\beta}^*$, $c_{\beta,1}^*$ are provided by \eqref{e:stima c0alfa1}, \eqref{e:stima calphagamma1}, \eqref{e:stima gammabeta}, \eqref{e:stima betauno}, respectively.
\end{proposition}

\begin{proof}
The proof analyzes the restriction of problem \eqref{e:P} to the four above intervals.

\paragraph{\em Case $[0,\alpha]$.} For $\phi\in[0,\alpha]$ we define
$$
h_1(\phi)=:-\dot{f}(-\phi+\alpha), \quad D_1(\phi)=:D(-\phi+\alpha),\quad g_1(\phi)=:-g(-\phi+\alpha).
$$
We also define $w(\phi):= z(-\phi+\alpha)$ and $\tilde c_1:=-c$. Then, when restricted to the interval $[0,\alpha]$, problem \eqref{e:P} is equivalent to
\begin{equation}\label{e:S1}
\begin{cases}
\dot{w}=h_1-\tilde c_1- D_1g_1/w \ &\mbox{ in } \ (0, \alpha),\\
w<0 \ &\mbox{ in } \ (0, \alpha),\\
w(0)=w(\alpha)=0.
\end{cases}
\end{equation}
Lemma \ref{lem:0101} applies with $\sigma_1=0$, $\sigma_2 = \alpha$ and $Q=D_1g_1$: since $Q$ is differentiable in $0$, then Lemma \ref{lem:0101} provides a threshold $\tilde c^*_{0, \alpha}$ such that  \eqref{e:S1} is solvable iff $\tilde c_1\ge \tilde c^*_{0, \alpha}$, i.e. $c\le -\tilde c_{0,\alpha}^*=: c_{0,\alpha}^*$. By \eqref{e:estimc12} we obtain
\begin{equation}
\label{e:stimaBase}
\max\left\{
\sup_{(0,\alpha]}\delta(f_1,0), h_1(0)+2\sqrt{\dot{D}_1(0)g_1(0)}\right\}
\le \tilde c^*_{0, \alpha} \le
\sup_{(0, \alpha]}\delta(f_1,0) + 2\sup_{(0,\alpha]}\sqrt{\Delta(D_1g_1,0)},
\end{equation}
with
$$
f_1(\phi)=\int_{0}^{\phi}h_1(\sigma)\, d\sigma=\int_{0}^{\phi}-\dot{f}(-\sigma+\alpha)\, d\sigma=-\int_{\alpha-\phi}^{\alpha}\dot{f}(s)\, ds=f(\alpha-\phi)-f(\alpha),
$$
whence
$$
\sup_{s\in (0, \alpha]}\delta(f_1,0)(s)= \sup_{s\in (0, \alpha]}\frac{f_1(s)}{s}=\sup_{s\in [0, \alpha)}\frac{f(s)-f(\alpha)}{\alpha-s}.
$$
Moreover we have
$$
\int_0^s \frac{D_1(\phi)g_1(\phi)}{\phi}\, d\phi=\int_0^s -\frac{D(-\phi+\alpha)g(-\phi+\alpha)}{\phi}\, d\phi=\int_{\alpha-s}^{\alpha}-\frac{D(\sigma)g(\sigma)}{\alpha-\sigma} \, d\sigma.
$$
Then formula \eqref{e:stimaBase} can be written as
\begin{equation}\label{e:stima c0alfa}
\max\left\{\sup_{[0, \alpha)}\left[-\delta(f,\alpha)\right],
-\dot{f}(\alpha) +2\sqrt{\dot{D}(\alpha)g(\alpha)}\right\} \le \tilde c^*_{0, \alpha}\le
 \sup_{[0, \alpha)}\left[-\delta(f,\alpha)\right] +2\sup_{[0, \alpha)}\sqrt{\Delta(Dg,\alpha)}.
\end{equation}
Hence,
\begin{equation}\label{e:stima c0alfa1}
\inf_{[0, \alpha)}\delta (f,\alpha)-2\sup_{[0, \alpha)}\sqrt{\Delta(Dg,\alpha)}
\le c^*_{0, \alpha}\le
\min\left\{\inf_{[0, \alpha)}\delta (f,\alpha),
\dot f(\alpha) -2\sqrt{\dot{D}(\alpha)g(\alpha)}\right\}.
\end{equation}

\paragraph{\em Case $[\alpha, \gamma]$.} We denote
$$
h_2(\phi):=-\dot{f}(\phi), \quad D_2(\phi):=-D(\phi), \quad g_2(\phi)=:-g(\phi).
$$
We also define $w(\phi): = -z(\phi)$ and $c_2:=-c$. Then problem \eqref{e:P}, when restricted to the interval $[\alpha,\gamma]$, becomes
 \begin{equation}\label{e:S2}
\begin{cases}
\dot{w}=h_2-c_2- D_2g_2/w \ &\mbox{ in } \ (\alpha, \gamma),\\
w<0 \ &\mbox{ in } \ (\alpha, \gamma],\\
w(\alpha)=0.
\end{cases}
\end{equation}
By Lemma \ref{lem:0101} we deduce the existence of a threshold $\tilde c^*_{\alpha, \gamma}$ such that \eqref{e:S2} is solvable iff $c_2\ge \tilde c^*_{\alpha, \gamma}$, i.e. $c\le -\tilde c^*_{\alpha, \gamma}=: c^*_{\alpha, \gamma}$. Moreover,
by \eqref{e:estimc12} we deduce
\[
\max\left\{
\sup_{(\alpha, \gamma]}\delta(f_2,\alpha), h_2(\alpha)+2\sqrt{\dot {D}_2(\alpha)g_2(\alpha)}
\right\}
\le \tilde c^*_{\alpha, \gamma}
\le
\sup_{(\alpha, \gamma]}\delta(f_2, \alpha)+2\sup_{(\alpha, \gamma]}\sqrt{\Delta(D_2g_2,\alpha)},
\]
where
$$
f_2(\phi):=\int_{\alpha}^{\phi} h_2(\sigma)\, d\sigma=-f(\phi)+f(\alpha), \, \phi \in [\alpha, \gamma].
$$
Whence, by returning to the variables $h,D,g$, we find
\begin{equation}\label{e:stima calphagamma}
\max\left\{ \sup_{(\alpha, \gamma]}\left\{-\delta(f,\alpha)\right\},
- \dot{f}(\alpha)+2\sqrt{\dot {D}(\alpha)g(\alpha)}\right\} \le \tilde c^*_{\alpha, \gamma}
\le
\sup_{(\alpha, \gamma]}\left\{-\delta(f,\alpha)\right\}+2\sup_{(\alpha, \gamma]}\sqrt{\Delta(Dg,\alpha)}.
\end{equation}
Hence,
\begin{equation}\label{e:stima calphagamma1}
\inf_{(\alpha, \gamma]}\delta(f,\alpha) - 2\sup_{(\alpha, \gamma]}\sqrt{\Delta(Dg,\alpha)}
\le
 c^*_{\alpha, \gamma}
\le
\min\left\{ \inf_{(\alpha, \gamma]}\delta(f,\alpha),
 \dot{f}(\alpha)-2\sqrt{\dot {D}(\alpha)g(\alpha)}\right\}.
\end{equation}
\paragraph{\em Case $[\gamma, \beta]$.} For $\phi\in[\gamma, \beta]$ we define
$$
h_3(\phi):=\dot{f}(-\phi+\gamma+\beta), \quad D_3(\phi)=:-D(-\phi+\gamma+\beta),\quad g_3(\phi)=:g(-\phi+\gamma+\beta).
$$
We also denote $w(\phi):= -z(-\phi+\gamma+\beta)$. Then in the interval $[\gamma,\beta]$ problem \eqref{e:P} can be written as
\begin{equation}\label{e:S3}
\begin{cases}
\dot{w}=h_3-c- D_3g_3/w \ &\mbox{ in } \ (\gamma, \beta),\\
w<0 \ &\mbox{ in } \ [\gamma, \beta),\\
w(\beta)=0.
\end{cases}
\end{equation}
By Lemma \ref{lem:0101} problem \eqref{e:S3} is solvable iff $c\ge c_{\gamma,\beta}^*$, for some threshold
$c^*_{\gamma, \beta}$. Upper and lower estimates for $c_{\gamma,\beta}^*$ can be obtained, as in the previous cases, by applying \eqref{e:estimc12}. In conclusion we find the estimates
\begin{equation}\label{e:stima gammabeta}
\max\left\{ \sup_{[\gamma, \beta)}\delta(f,\beta),
\dot{f}(\beta)+2\sqrt{\dot {D}(\beta)g(\beta)}\right\} \le c^*_{\gamma, \beta}
\le
\sup_{[\gamma, \beta)}\delta(f,\beta) + 2\sup_{[\gamma, \beta)}\sqrt{\Delta(Dg,\beta)}.
\end{equation}

\paragraph{\em Case  $[\beta, 1]$.} In this case we directly apply Lemma \ref{lem:0101}: the problem
   \begin{equation}\label{e:S4}
\begin{cases}
\dot{z}=\dot{f}-c- Dg/z \ &\mbox{ in } \ (\beta, 1),\\
z<0 \ &\mbox{ in } \ (\beta, 1),\\
z(\beta)=z(1)=0,
\end{cases}
\end{equation}
is solvable iff $c\ge c^*_{\beta, 1}$, for some $c_{\beta,1}^*$. Again, estimates for $c^*_{\beta, 1}$ are deduced by \eqref{e:estimc12}:
\begin{equation}\label{e:stima betauno}
\max\left\{\sup_{(\beta, 1]}\delta(f,\beta),
\dot{f}(\beta)+2\sqrt{\dot {D}(\beta)g(\beta)}\right\} \le c^*_{\beta, 1}
\le
\sup_{(\beta, 1]}\delta(f,\beta) + 2\sup_{(\beta, 1]}\sqrt{\Delta(Dg,\beta)}.
\end{equation}

\noindent This concludes the analysis of the restrictions of problem \eqref{e:P} to the four above intervals. Condition $c_1\le c_0$  is the requirement that there is a common admissible speed $c$ for the above subproblems. In this case $c \in [c_1, c_0]$.
\end{proof}

\begin{remark}\label{r:phi-dec}
{\rm
Since $D$ and $g$ vanish in the interior of none of the above sub-intervals, one finds $\phi' <0$ if $\phi \in (0,1) \setminus\{\alpha, \gamma, \beta\}$ (see \cite[Proposition 3.1(ii)]{BCM3}). Moreover, by \cite[Theorem 2.9 (i)]{CM-DPDE}, we deduce that the profile never reaches the value $1$ for a finite value of $\xi$; the same result holds for the value $0$, by exploiting again \cite[Theorem 2.9 (i)]{CM-DPDE} after the change of variables that led to \eqref{e:S1}. At last, we have $\phi'(\gamma)<0$ by the second part of the proof of Proposition 3.1(ii) in \cite{BCM3}. As a consequence, the profile $\phi$ is strictly monotone.
}
\end{remark}

\begin{proofof}{Theorem \ref{t:solvability}}
The proof follows an argument based on the reduction of \eqref{e:E}-\eqref{e:bvs} to \eqref{e:P}, see \cite{BCM3}.

\smallskip

First, assume $c_1<c_0$. We argue separately in the four sub-intervals where $Dg \neq 0$ and then we put together what we found. Thus, let $z$ be the solution of \eqref{e:P} associated to some $c\in [c_1, c_0]$. Define $\phi_{1,\beta}$, $\phi_{\beta,\gamma}$, $\phi_{\gamma, \alpha}$ and $\phi_{\alpha,0}$ as the solutions of
\begin{equation}
\label{e:equation}
\phi'=\frac{z(\phi)}{D(\phi)},
\end{equation}
with the initial data (respectively)
\[
\phi_{1,\beta}(0)=\frac{1+\beta}{2}, \quad \phi_{\beta,\gamma}=\frac{\beta+\gamma}{2},\quad \phi_{\gamma, \alpha}(0)=\frac{\gamma+\alpha}{2}, \quad \phi_{\alpha,0}(0)=\frac{\alpha}{2}.
\]
Since the right-hand side of \eqref{e:equation} is locally of class $C^1$, then $\phi_{1,\beta}$, $\phi_{\beta,\gamma}$, $\phi_{\gamma, \alpha}$ and $\phi_{\alpha,0}$, exist and are unique in their respective maximal existence intervals.

We focus on the pasting of $\phi_{1,\beta}$, $\phi_{\beta,\gamma}$ at $\beta$. Let $\phi_{1,\beta}$, $\phi_{\beta,\gamma}$ be maximally defined in $(\xi_1, \xi_\beta^1) \subset \R$, $(\xi_\beta^2, \xi_\gamma^1)\subset \R$, with
\[
-\infty \le \xi_1 <0 < \xi_\beta^1 \le \infty, \quad -\infty \le \xi_\beta^2 < 0 < \xi_\gamma^1 \le \infty,
\]
and satisfying
\begin{gather*}
\lim_{\xi \to \xi_1^+} \phi_{1,\beta}(\xi)=1, \ \lim_{\xi \to \{\xi_\beta^1\}^-}\phi_{1,\beta}(\xi)=\beta,
\quad
\hbox{ and }
\quad
\lim_{\xi \to \{\xi_\beta^2\}^+} \phi_{\beta,\gamma}(\xi)=\beta \ \lim_{\xi \to \{\xi_\gamma^1\}^-}\phi_{\beta,\gamma}(\xi)=\gamma.
\end{gather*}
In order to glue together $\phi_{1,\beta}$ and $\phi_{\beta,\gamma }$ (after space shifts), we need to prove $\xi_\beta^1 \in \R$ and $\xi_\beta^2 \in \R$.
We have
\[
\lim_{\xi \to \{\xi_\beta^2\}^+}\phi_{\beta,\gamma}'(\xi)=\lim_{\xi \to \{\xi_\beta^2\}^+}\frac{z\left(\phi_{\beta,\gamma}(\xi)\right)}{D\left(\phi_{\beta,\gamma}(\xi)\right)}=\lim_{s\to \beta^-}\frac{z(s)}{D(s)}=\lim_{t\to \gamma^+}\frac{w(t)}{D_3(t)},
\]
with $w$ and $D_3$ as in \eqref{e:S3}. The last limit is essentially discussed in the proof of \cite[Theorem 2.5]{CM-DPDE};  the only difference is that the interval  $[0, 1]$ appearing there is now replaced by   $[\gamma, \beta]$. Reasoning as there we obtain that
\[
\lim_{t\to \gamma^+}\frac{w(t)}{D_3(t)}\in [-\infty, 0);
\]
hence  $\xi_\beta^2$ is a real value.
%
With a similar reasoning, this time directly applied to $z(\phi_{1,\beta})$ and $D(\phi_{1,\beta})$, we can prove that also $\xi_\beta^2$ is a real value.

The remaining pastings are exactly proved as in the proof of \cite[Proposition 3.2]{BCM3} and we refer the reader to that paper for details. To this aim, in particular, we need that  $z(\gamma)>0$, which is satisfied when $c_1<c_0$ by Proposition \ref{prop:solvability}.  The proof of the first statement is complete.

\smallskip

We now prove the second statement. Suppose that  \eqref{e:E}-\eqref{e:bvs} admits a profile $\phi$ associated to some speed $c\in \R$. In particular, $\phi$ is decreasing and hence it can be decomposed into sub-profiles $\phi_{1,\beta}$, $\phi_{\beta,\gamma}$, $\phi_{\gamma,\alpha}$ and $\phi_{0,\beta}$ connecting, respectively, $\beta$ to $1$, $\gamma$ to $\beta$, and so on. By Remark \ref{r:phi-dec} we have $\phi' <0$ if $\phi \in (0,1) \setminus\{\alpha, \gamma, \beta\}$. Therefore, $\phi_{1,\beta}$ is invertible for $\phi_{1,\beta} \in (\beta,1)$, $\phi_{\beta,\gamma}$ is invertible for $\phi_{\beta,\gamma} \in (\gamma, \beta)$, and so on. Let $\zeta=\zeta(\phi): (\beta,1) \to \R$ be the inverse function of $\phi_{1,\beta}$, and set
\[
z(\phi):=D(\phi)\phi_{1,\beta}'\left(\zeta(\phi)\right), \ \mbox{ for } \ \phi \in (\beta,1).
\]
By direct computations, the function $z$ solves $\eqref{e:P}_1$ in $(\beta,1)$ (where $z \in C^1$) and, by adapting \cite[Lemma 3.1]{BCM3}, it can be extended to a function of class $C^0[\beta,1]$, still called $z$. Also, as in \cite[Lemma 3.1]{BCM3}, we have $z(1)=0$. Arguing similarly in the other sub-intervals, one finds that $z \in C^0[0,1]$ is in $C^1$ in $(0,\alpha)\cup (\alpha, \beta) \cup (\beta,1)$ and it satisfies \eqref{e:P}. For more details we refer to the similar case presented in the proof of \cite[Proposition 3.1 {\em (ii)}]{BCM3}, which applies because $g$ satisfies \cite[(2.2)]{BCM3}. According to Proposition \ref{prop:solvability} we obtain $c_1\le c_0$ and then also the second statement is proved.
\end{proofof}

\begin{remark}\label{r:exw-fconcave}
{\rm We now provide a simple argument showing why wavefronts should exist for suitable concave $f$, in the case the drift $\dot f$ is first positive and then negative. For $\lambda>0$, let $f$ be defined by $\lambda u$ in $(0,\gamma)$ and $-\lambda(u-2\gamma)$ in $(\gamma,1)$, so that $f$ is Lipschitz continuous with $\dot f=\lambda$ in $(0,\gamma)$ and $\dot f=-\lambda$ in $(\gamma,1)$. In this case, the role of $\lambda$ is to shift to the right (of magnitude $+\lambda$) the estimates for $c_0$, as \eqref{e:stima c0alfa1} and \eqref{e:stima calphagamma1} show, and to shift to the left (of $-\lambda$) the estimates for $c_1$ (see \eqref{e:stima gammabeta} and \eqref{e:stima betauno}). Hence, \eqref{e:stima} holds true for $\lambda$ large enough.
}
\end{remark}

We denote by $s_{0,\alpha}$, $s_{\alpha, \gamma}$, $s_{\gamma, \beta}$, $s_{\beta, 1}$, the lower bounds in \eqref{e:stima c0alfa1}, \eqref{e:stima calphagamma1}, \eqref{e:stima gammabeta}, \eqref{e:stima betauno}, respectively, and with
$\Sigma_{0,\alpha}$, $\Sigma_{\alpha, \gamma}$, $\Sigma_{\gamma, \beta}$, $\Sigma_{\beta, 1}$,
the corresponding upper bounds. In other words we rewrite \eqref{e:stima c0alfa1}, \eqref{e:stima calphagamma1}, \eqref{e:stima gammabeta}, \eqref{e:stima betauno} as
\begin{equation}
\label{e:sub-estimates}
s_{0,\alpha}\le c^*_{0,\alpha}\le \Sigma_{0,\alpha}, \quad s_{\alpha, \gamma}\le c^*_{\alpha, \gamma}\le \Sigma_{\alpha, \gamma}, \quad s_{\gamma, \beta}\le c^*_{\gamma, \beta}\le \Sigma_{\gamma, \beta}, \quad s_{\beta, 1}\le c^*_{\beta, 1}\le \Sigma_{\beta, 1}.
\end{equation}
Define moreover
\[
s_{0,\gamma}:=\inf_{[0,\gamma]}\delta(f,\alpha) - 2\sup_{[0, \gamma]}\sqrt{\Delta(Dg,\alpha)}
\quad\hbox{ and }\quad
\Sigma_{\gamma, 1}:=\sup_{[\gamma, 1]}\delta(f,\beta) + 2\sup_{[\gamma, 1]}\sqrt{\Delta(Dg,\beta)}.
\]
Here above, the arguments of the supremums are not defined at $\alpha$ and $\beta$, respectively; of course, since $f,D,g\in C^1$, we understand them as $-\dot f (\alpha), \, \dot D(\alpha)g(\alpha), \, \dot f (\beta)$ and $\dot D(\beta)g(\beta)$, respectively. Under this notation we immediately deduce the following result.

\begin{lemma}\label{cor:key}
If $\Sigma_{\gamma, 1} < s_{0, \gamma}$,
then condition \eqref{e:stima}  is satisfied.
\end{lemma}

\begin{proof}
According to the right-hand sides of the estimates \eqref{e:stima gammabeta} and \eqref{e:stima betauno} we have
$c_1\le \max\{\Sigma_{\gamma, \beta}, \, \Sigma_{\beta, 1}\} \le \Sigma_{\gamma,1}$.
By the assumption $\Sigma_{\gamma, 1} < s_{0, \gamma}$ we obtain
\[
c_1\le  \Sigma_{\gamma,1}< s_{0, \gamma}
\le \min\{s_{0, \alpha}, \, s_{\alpha, \gamma}\}.
\]
Because of \eqref{e:stima c0alfa1} and \eqref{e:stima calphagamma1} we deduce $c_1<\min\{s_{0, \alpha}, \, s_{\alpha, \gamma}\}\le \min\{c^*_{0, \alpha}, c^*_{\alpha, \gamma} \}=c_0$.
\end{proof}

\begin{proofof}{Corollary \ref{nec}}
If a wavefront exists, then necessarily $c_1 \le c_0$ because of Theorem \ref{t:solvability}. Then, by \eqref{e:stima-def} and \eqref{e:sub-estimates} it follows
\begin{multline}
\max \left\{\sup_{[\gamma,1]}\delta(f,\beta), \dot{f}(\beta)+2\sqrt{\dot{D}(\beta)g(\beta)}\right\} =
\\
\max\{s_{\gamma,\beta}, \, s_{\beta,1}\} \le c_1 \le c_0 \le  \min\{\Sigma_{0,\alpha},\, \Sigma_{\alpha,\gamma}\} =
\\
\min\left\{\inf_{[0,\gamma]}\delta(f,\alpha), \dot{f}(\alpha)-2\sqrt{\dot{D}(\alpha)g(\alpha)}\right\},
\label{e:2-0}
\end{multline}
which is \eqref{nec0}.
\end{proofof}

\begin{proofof}{Corollary \ref{t:CS}}
First, notice that \eqref{e:ssigma} is exactly $\Sigma_{\gamma, 1} < s_{0, \gamma}$ after trivial manipulations. As a consequence, Theorem \ref{t:solvability} and Corollary \ref{cor:key}. imply the existence of wavefronts.

To obtain \eqref{e:positive speeds}, we impose $\max\{s_{\gamma,\beta}, \, s_{\beta,1}\} >0$ in \eqref{e:2-0}, which in turn implies that $c_1>0$; notice that the left-hand side in \eqref{e:2-0} is precisely the left-hand side of \eqref{e:positive speeds}. Analogously, to obtain \eqref{e:negative speeds}, we impose $\min\{\Sigma_{0,\alpha},\, \Sigma_{\alpha,\gamma}\} <0$. This implies $c_0 <0$.
\end{proofof}

We now investigate when the set $\mathcal{J}$ of admissible speeds contains positive values.

\begin{lemma} \label{l:sooof} Assume
\begin{equation}\label{e:c>0}
\inf_{[0, \gamma]}\delta(f, \alpha)>2\sup_{[0, \gamma]}\sqrt{\Delta(Dg, \alpha)}.
\end{equation}
Then either $\mathcal{J}=\emptyset$ or $\mathcal{J}\cap (0, +\infty)\ne \emptyset$.
\end{lemma}
\begin{proof}
Assume that $\mathcal{J}\ne \emptyset$; then, according to Theorem \ref{t:solvability}, we have $\mathcal{J}\cap (0, \infty)\ne \emptyset$ if and only if $c_0>0$. By \eqref{e:stima c0alfa1} and \eqref{e:stima calphagamma1}, we have
\[
\begin{aligned}
c_0&=\min\{c_{0,\alpha}^*, c_{\alpha,\gamma}^*\}
\\
&\ge \min\big\{\inf_{[0,\alpha)}\delta(f,\alpha)-2\sup_{[0,\alpha)}\sqrt{\Delta(Dg, \alpha)}, \inf_{(\alpha,\gamma]}\delta(f,\alpha)-2\sup_{(\alpha,\gamma]}\sqrt{\Delta(Dg, \alpha)}\big\}
\\
&\ge \inf_{[0,\gamma]}\delta(f,\alpha)-2\sup_{[0,\gamma]}\sqrt{\Delta(Dg, \alpha)}=s_{0,\gamma}.
\end{aligned}
\]
If condition \eqref{e:c>0} is satisfied, then $c_0>0$.
\end{proof}

\section{Existence of wavefronts in the model of biased movements}
\label{s:app}
\setcounter{equation}{0}
In this section we investigate the presence of wavefronts to the biased model \eqref{e:model} and prove their main qualitative properties. We make use of the results provided in Section \ref{s:tools} for a general reaction-diffusion-convection process.

\begin{proofof}{Lemma \ref{l:la}}
The function $D$ in \eqref{e:coeffs D} is a parabola with $D(0)=D_i$ and $D(1)=D_g$. We have $\dot D(u) =-(D_i-D_g)(4-6u)$, which vanishes iff $u=\frac23$, and $D(\frac23)=\frac13(-D_i+4D_g)$. Then $D$ is positive-negative-positive if and only if $D_i>4D_g$; the case $D_g=0$ is excluded because then $D$ changes sign only once in $(0,1)$. Then the two zeros $\alpha$ and $\beta$ of $D$ satisfy $\eqref{e:alphabeta}_{1,2}$.
Moreover $g(0)=0$ and $g(1)=0$ if and only if $k_g=0$; under this assumption, $g$ also vanishes at $\gamma$ defined in \eqref{e:gamma}. Hence $g$ satisfies condition (g) if and only if $k_g=0, \, \lambda_g>0$ and $r_i>0$. The condition $\gamma\in(\alpha,\beta)$ is then equivalent to \eqref{e:restr1-3}.
\end{proofof}

In the following proofs, we often make use of the notation
\begin{equation}\label{e:notation}
p:= C_iD_i+C_gD_g \quad \text{and} \quad q:=C_gD_g.
\end{equation}
We now rewrite formulas \eqref{e:coeffs f}--\eqref{e:coeffs g} by exploiting \eqref{e:alphabeta}, \eqref{e:gamma} and \eqref{e:notation}:
\begin{align}
f(u)&= -pu(1-u)^2 -qu(1-u),
\label{e:Df1}
\\
D(u) & = 3(D_i-D_g)(u-\alpha)(u-\beta),
\label{e:Df2}
\\
g(u)&= (r_i+\lambda_g)\cdot u(1-u)(u-\gamma).
\label{e:Df3}
\end{align}

\begin{remark}
{\rm We point out that $\dot f(0)=-(p+q)$ and $\dot f(1)=q$; these quantities can be understood the drift at very low and maximum concentration, respectively.
}\end{remark}

\begin{remark}
{\rm The movement velocity $v=v(u)$ is defined by $f(u)=: uv(u)$. Then
$v(u) = -pu^2 + (2p+q)u-(p+q)=(1-u) \left(pu-(p+q)\right)$,
and then $v$ vanishes at the maximum density $1$; it can  also possibly vanish at
$u_0=\frac{p+q}{p}$ (i.e., if $u_0\in[0,1)$).
This is analogous to similar models in collective movements \cite[\S 3.1]{Whitham}.
Recalling that $q<0$, it is easy to see that only the following cases may occur (for simplicity we do not include the case $p+q=0$, when $u_0=0$, or $p=0$, when $u_0$ is missing, for which slightly different results hold):

\begin{enumerate}

\item $q<0<p+q$. Then $v$ is concave, it is first negative, then positive; $f$ is convex-concave.
\item $p+q<0<p$. Then $v$ is positive and concave; $f$ is concave or convex-concave.
\item $p<0$, $q<0$. Then $v$ is positive and convex; $f$ is concave or concave-convex.
\end{enumerate}
}
\end{remark}

We assume conditions \eqref{e:restr1-1}--\eqref{e:restr1-3}; in particular, the assumption $C_g<0$ becomes $q<0$.
Under this notation, for $\phi, \phi_0 \in (0,1)$ we have, see \eqref{e:dq},
\begin{equation}\label{e:deltaphiphi0}
\begin{aligned}
\delta(f, \phi)(\phi_0)
&= -(p+q)+(2p+q)(\phi + \phi_0)-p(\phi^2+\phi\phi_0 +\phi_0^2).
\end{aligned}
\end{equation}

\begin{proofof}{Proposition \ref{p:CN}}
We apply condition \eqref{nec1}. Since $\dot{f}(u)=-p(3u^2-4u+1)+q(2u-1)$,
then \eqref{nec1} applies if
$-p(3\alpha^2-4\alpha+1)+q(2\alpha-1) > -p(3\beta^2-4\beta+1)+q(2\beta-1)$,
that is
\begin{equation}
\label{e:nec1}
-p\left(3(\alpha^2-\beta^2)-4(\alpha-\beta)\right)+2q(\alpha-\beta)>0.
\end{equation}
By \eqref{e:alphabeta}
we obtain that \eqref{e:nec1} is equivalent to
$2q (\alpha-\beta)=-\frac{4q}{3}\omega >0$, that is, $q<0$.
Hence, we deduce $C_g<0$ since $D_g>0$ by Lemma \ref{l:la}.
\end{proofof}

A sufficient condition for the existence of wavefronts to equation \eqref{e:model} is \eqref{e:ssigma}. The following result provides an upper estimate of the right-hand side of \eqref{e:ssigma}.

\begin{lemma}\label{l:stimaDx} We have
\begin{equation*}
2\sup_{[0, \gamma]}\sqrt{\Delta(Dg, \alpha)}+2\sup_{[\gamma,1]}\sqrt{\Delta(Dg, \beta)}
\le
\sqrt{\frac{D_g}{\lambda_g}}\sqrt{d-1}\left( \sqrt{\mu(2+\omega)} + \sqrt{1+\omega}\right).
\end{equation*}
\end{lemma}
\begin{proof}
By \eqref{e:Df2} we have
\(
D(\phi)g(\phi)=3(D_i-D_g)(r_i+\lambda_g)\phi(\phi-\alpha)(\phi-\gamma)(\phi-\beta)(1-\phi).
\)
Then, for $\phi \in [0,\gamma]$ we obtain
\begin{align}
\frac{D(\phi)g(\phi)}{\phi-\alpha}
&\le \frac 34 (D_i-D_g)(r_i+\lambda_g)(\phi-\gamma)(\phi-\beta)\nonumber
\le \frac 34 (D_i-D_g)(r_i+\lambda_g)\gamma\beta
\nonumber\\
&=\frac 14 r_i(D_i-D_g)(2+\omega).\label{e:plij}
\end{align}
By \eqref{e:deltaDelta} and \eqref{e:plij} we deduce
\begin{equation}\label{e:dx1}
2\sup_{[0, \gamma]}\sqrt{\Delta(Dg, \alpha)}\le \sqrt{D_g}\sqrt{r_i(d-1)(2+\omega)}.
\end{equation}
With a similar reasoning we have that
\begin{equation}\label{e:dx2}
2\sup_{[\gamma,1]}\sqrt{\Delta(Dg, \beta)}\le \sqrt{D_g}\sqrt{\lambda_g(d-1)(1+\omega)},
\end{equation}
since we have $(\phi-\gamma)(\phi-\alpha)\le (1-\gamma)(1-\alpha)$, for $\phi \in [\gamma,1]$, because $(r_i+\lambda_g)(1-\gamma) = \lambda_g$. We complete the proof by combining \eqref{e:dx1} and \eqref{e:dx2}.
\end{proof}

\subsection{A strictly concave convective term}\label{ss:concave}
The left-hand side of \eqref{e:ssigma} takes a simple form when $f$ is strictly concave (see Remark \ref{r:lhswith f concava}); for this reason we first consider this case, see Figure \ref{f:Dgfsconcave}. The following result characterizes the strict concavity of the function $f$, see \eqref{e:restr2}.

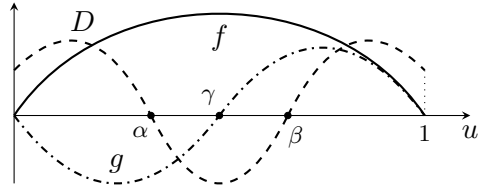
\begin{figure}[htb]
\begin{center}
\begin{tikzpicture}[>=stealth, scale=0.6]
\draw[->] (0,0) --  (10,0) node[below]{$u$} coordinate (x axis);
\draw[->] (0,0) -- (0,2.5) coordinate (y axis);
\draw (0,0) -- (0,-1.5);
\draw[thick,dashed] (0,1) .. controls (1,2) and (2,2) .. (3,0) node[left=4, below=0]{{\footnotesize{$\alpha$}}} node[midway,above]{$D$};
\draw[thick,dashed] (3,0) .. controls (4,-2) and (5,-2) .. (6,0);
\draw[thick,dashed] (6,0) node[right=3, below=0]{{\footnotesize{$\beta$}}} .. controls (7,2) and (8,2) .. (9,1);
\draw[dotted] (9,1) -- (9,0) node[below]{{\footnotesize{$1$}}};
\draw[thick,dashdotted] (0,0) .. controls (1.5,-2) and (3,-2) .. (4.5,0) node[left=4, above=0]{{\footnotesize{$\gamma$}}} node[midway,above]{$g$};
\draw[thick,dashdotted] (4.5,0) .. controls (6,2) and (7.5,2) .. (9,0);
\filldraw[black] (3,0) circle (2pt);
\filldraw[black] (4.5,0) circle (2pt);
\filldraw[black] (6,0) circle (2pt);
\draw[thick] (0,0) .. controls (2,3) and (7,3) .. (9,0) node[midway,below]{$f$};
\end{tikzpicture}

\end{center}
\caption{\label{f:Dgfsconcave}{Plots of the functions $D$ (dashed line), $g$ (dashdotted line) and $f$ (solid line) in the case $f$ is strictly concave.}}
\end{figure}

\begin{lemma}
\label{lem:conc}
The function $f$ in \eqref{e:coeffs f} is strictly concave if and only if \eqref{e:restr2} is satisfied.
\end{lemma}
\begin{proof}
By \eqref{e:Df1} we compute
$\ddot f(u)=-6pu+4p+2q$; therefore $\ddot f<0$ in $(0,1)$ if and only if
\begin{equation}\label{e:boh}
-3pu+2p+q<0, \ \mbox{ for any } \ u\in (0,1).
\end{equation}
The line $-3pu+2p+q=0$ connects the points $(0,2p+q)$ and $(1,-p+q)$. We remark that $2p+q=-p+q=0$ is not possible since $q<0$ by conditions \eqref{e:restr1-1}--\eqref{e:restr1-3} and \eqref{e:notation}. Hence, \eqref{e:boh} holds if and only if
\begin{equation}
\label{e:boh2}
\begin{cases}
2p+q\le 0,\\
-p+q \le 0.
\end{cases}
\end{equation}
Conditions $\eqref{e:boh2}$ hold if and only if $q\le p\le -q/2$, which is equivalent to \eqref{e:restr2}.
\end{proof}

\begin{remark}\label{rem:convf}
{\rm
From the proof of Lemma \ref{lem:conc} we deduce that $f$ is strictly convex iff $-\frac{C_gD_g}{2}\le C_iD_i + C_gD_g\le C_gD_g$; this condition does not match with the assumption $C_g<0$, which is necessary to have wavefronts to equation \eqref{e:model} satisfying \eqref{e:bvs} by Proposition \ref{p:CN}. Indeed, the bare convexity of $f$ in $[\alpha,\beta]$ is sufficient to hinder the existence of such wavefronts, because the right-hand side of \eqref{nec1} is strictly positive when $D$ and $g$ are as in \eqref{e:Df2}, \eqref{e:Df3}.
}
\end{remark}
We now apply the sufficient condition \eqref{e:ssigma} to the current case.
%

\begin{theorem}\label{t:mainmodel} If $f$ is strictly concave and
\begin{equation}
\frac{d-1}{\sqrt{d-4}} \frac{\sqrt{\mu (2+\omega)} + \sqrt{1+\omega}}{2\mu + 5+\c d(\mu-2)}
(\mu+1) < \frac29 E_g
\label{e:restr3}
\end{equation}
holds, then equation \eqref{e:model} admits wavefronts satisfying condition \eqref{e:bvs}.
\end{theorem}

\begin{proof}
In order to apply \eqref{e:ssigma} we exploit Remark \ref{r:lhswith f concava}.
Then, by exploiting \eqref{e:deltaphiphi0}, we compute
\begin{equation*}
\begin{aligned}
\delta(f,\alpha)(\gamma)-\delta(f,\beta)(\gamma)&=(2p+q)(\alpha-\beta)-p(\alpha^2+(\alpha -\beta)\gamma -\beta^2)
\\
&=(\beta-\alpha)\left(p(\alpha+\beta - 2+\gamma)-q \right),
\end{aligned}
\end{equation*}
whence, from $\beta-\alpha=\frac 23\omega$ and $\alpha+\beta=\frac 43$, we get
\begin{equation}\label{e:lefthandside}
\inf_{[0, \gamma]}\delta(f, \alpha)-\sup_{[\gamma, 1]}\delta(f, \beta)=\frac23 \omega\left[p\left(\gamma-\frac23\right)-q\right].
\end{equation}
%
By \eqref{e:notation} and \eqref{e:gamma} we can write
\begin{align*}
p\left(\gamma -\frac23\right)-q&=C_gD_g\left(\frac{r_i}{r_i+\lambda_g}-\frac53\right)+
C_iD_i\left(\frac{r_i}{r_i+\lambda_g}-\frac23\right)
\\
&=\frac{1}{3(r_i+\lambda_g)}\left(C_gD_g(-2r_i-5\lambda_g)+C_iD_i(r_i-2\lambda_g) \right).
\end{align*}
Therefore, when $f$ is strictly concave we have
\begin{equation*}
\begin{aligned}
\inf_{[0, \gamma]}\delta(f, \alpha)-\sup_{[\gamma, 1]}\delta(f, \beta)
=\frac{2\omega}{9(r_i+\lambda_g)}
&\left(C_gD_g(-2r_i-5\lambda_g)+C_iD_i(r_i-2\lambda_g) \right).
\end{aligned}
\end{equation*}
By the above formula, Lemma \ref{l:stimaDx}, and \eqref{e:omegadad}, condition \eqref{e:restr3}
implies \eqref{e:ssigma}.
\end{proof}

\begin{corollary}\label{c:lmmk}
Under \eqref{e:restr2}, condition \eqref{e:restr3} is satisfied if
\begin{equation}\label{e:lpoi}
\sqrt{\mu(2+\omega)+(1+\omega)} \frac{d-1}{\sqrt{d-4}} <\frac{4}{9\sqrt{2}} E_g.
\end{equation}
\end{corollary}
\begin{proof}
By \eqref{e:restr2} we have $2\mu+5 + sd(\mu-2)=(2+\c d)\mu + (5-2\c d) \ge 2(\mu + 1)$;
so, condition \eqref{e:restr3} holds if
\begin{align*}
\left(\sqrt{\mu(2+\omega)} + \sqrt{1+\omega}\right) \frac{d-1}{\sqrt{d-4}} < \frac49 E_g.
\end{align*}
In turn, this condition is satisfied if \eqref{e:lpoi} holds.
\end{proof}

\begin{remark}
\label{rem:specific1}
{\rm
For fixed $C_g, D_g$ satisfying \eqref{e:restr1-3}, condition \eqref{e:lpoi} (that is, \eqref{e:argen}) identifies the triangle $\mathcal{T}_g(d)$ in \eqref{e:triangleT}. Therefore, under \eqref{e:restr2}, if $(r_i,\lambda_g)\in \mathcal{T}_g(d)$ then the assumptions of Theorem \ref {t:mainmodel} are satisfied and
 equation \eqref{e:model} admits wavefronts satisfying \eqref{e:bvs}.
}
\end{remark}

We now investigate the sign of the speed of wavefronts; this issue is important in the biological framework. We find below conditions in order that wavefronts with {\em positive} speed exists and conditions assuring that every wavefront has {\em negative} speed.

About the case of {\em positive speeds}, by \eqref{e:deltaphiphi0}, \eqref{e:gamma}, \eqref{e:restr1-3} and \eqref{e:notation} we obtain
\begin{align}
&\inf_{[0, \gamma]}\delta(f, \alpha) =\frac{f(\gamma)-f(\alpha)}{\gamma-\alpha}=-(p+q)+(2p+q)(\alpha+\gamma )-p(\alpha^2+\alpha\gamma +\gamma^2)
\nonumber
\\
&=C_gD_g\left( -\frac49 -\frac59 \omega-\frac{\omega^2}{9}+\frac{7+\omega}{3}\gamma-\gamma^2\right)+C_iD_i\left( -\frac19 -\frac29 \omega-\frac{\omega^2}{9}+\frac{4+\omega}{3}\gamma-\gamma^2\right)
\nonumber
\\
&= \frac{|C_g|D_g}{9}\left((1-\c d)(\omega^2+9\,\gamma^2-3\,\omega\gamma)+(5-2\,\c d)\omega-3(7-4\,\c d)\gamma+4-\c d\right)
\nonumber
\\
&=: \frac{|C_g|D_g}{9}\tau(\omega,\gamma,\c d).
\label{e:tau}
\end{align}

Denote
\begin{equation}
\label{e:equa1}
\mathcal{R}:= \left\{(\omega,\gamma) \colon \sqrt{3}-1< \omega <1 \ \hbox{ and }\  \frac{2-\omega}{3}< \gamma < 1-\frac{1}{\sqrt{3}}\right\}.
\end{equation}

\begin{lemma}\label{lem:tau}
We have $\tau(\omega,\gamma,\c d)>0$ for every $(\omega,\gamma)\in\mathcal{R}$ and $0 \le \c d \le 3/2$.
\end{lemma}
\begin{proof}
First, it is easy to show that the function $\partial_{\c d}\tau (\omega,\gamma,\c d) = - \omega^2 -9\gamma^2 +3\omega\gamma -2\omega+12\gamma-1$ has no critical points in the triangle
\[
\mathcal{T}:=\left\{(\omega,\gamma)\in \mathbb{R}^2:\, 0<\omega <1 \ \mbox{ and } \  \frac{2-\omega}{3} < \gamma < \frac{2+\omega}{3}\right\},
\]
which contains the set $\mathcal{R}$, see Figure \ref{f:2triangles}.
Moreover, on $\partial\mathcal{T}$ we have $\partial_{\c d}\tau>0$ and then
\begin{equation}
\label{e:mono-cd}
\partial_{\c d}\tau (\omega,\gamma,\c d) >0 \quad \hbox{ for } (\omega, \gamma)\in\mathcal{T}.
\end{equation}
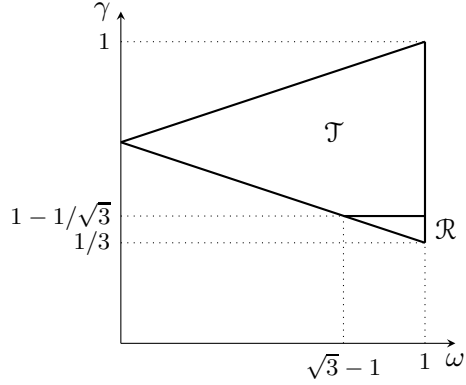
\begin{figure}[htb]
\begin{center}
\begin{tikzpicture}[>=stealth, scale=4]
\draw[->] (0,0) --  (1.1,0) node[below]{$\omega$} coordinate (x axis);
\draw[->] (0,0) -- (0,1.1) node[left]{$\gamma$} coordinate (y axis);

\draw[thick](0,2/3)--(1,1);
\draw[thick](0,2/3)--(1,1/3);
\draw[thick](1,1/3)--(1,1);
\draw[dotted](1,0) node[below]{\footnotesize$1$}--(1,1/3);
\draw[dotted](0,1) node[left]{\footnotesize$1$}--(1,1);
\draw[dotted](0,1/3) node[left]{\footnotesize$1/3$}--(1,1/3);
\draw(0.7,0.7) node{$\mathcal{T}$};

\draw[dotted](0,0.422) node[left]{\footnotesize$1-1/\sqrt{3}$}--(1,0.422);
\draw[dotted](0.732,0) node[below]{\footnotesize$\sqrt{3}-1$}--(0.732,0.422);
\draw[thick](0.732,0.422)--(1,0.422);
\draw(1,0.38) node[right]{$\mathcal{R}$};

\end{tikzpicture}

\end{center}
\caption{\label{f:2triangles}{The triangles $\mathcal{T}$ and $\mathcal{R}$.}}
\end{figure}

Then, by the monotonicity property proved in \eqref{e:mono-cd}, it is sufficient to prove that $\tau(\omega,\gamma,0)>0$ for every $(\omega,\gamma)\in\mathcal{R}$. We have
\[
\begin{aligned}
 \tau(\omega,\gamma,0)&=\omega^2+9\,\gamma^2-3\,\omega\gamma+5\omega-21\gamma+4=
 \left(\omega-\frac{3}{2}\gamma +\frac52\right)^2+\frac{27}{4}(1-\gamma)^2-9.
 \end{aligned}
 \]
 This quantity is positive if, in particular,
 \[
 \omega-\frac{3}{2}\gamma +\frac52>\frac{3\sqrt{3}}{2} \ \mbox{ and } \ 1-\gamma >\frac1{\sqrt{3}}.
 \]
The second inequality implies the first one when $\omega>\sqrt3-1$, and then $\tau(\omega, \gamma,0)>0$ for every $(\omega, \gamma) \in \mathcal{R}$.
\end{proof}

\begin{remark}\label{r:RR}
{\rm We easily see that $(\omega,\gamma)\in\mathcal{R}$ iff $(r_i,\lambda_g)\in\tilde{\mathcal{R}}(d)$ and $\sqrt{3}-1< \omega <1$, where
\[
\tilde{\mathcal{R}}(d) = \left\{(r_i,\lambda_g)\in\R^+\times\R^+ \colon \frac{1}{\sqrt{3}-1} < \frac{\lambda_g}{r_i}<\frac{1+\omega}{2-\omega} \right\}.
\]}
\end{remark}

\begin{theorem}\label{t:allk}
Assume $f$ is strictly concave, $(r_i,\lambda_g)\in\tilde{\mathcal{R}}(d)$ and $\sqrt{3}-1< \omega <1$.
If \eqref{e:ppoo} is satisfied, then either $\mathcal{J}=\emptyset$ or $\mathcal{J}\cap (0, +\infty)\ne \emptyset$.
\end{theorem}
\begin{proof}
By Remark \ref{r:RR}, Lemma \ref{lem:tau} applies and then $\tau(\omega,\gamma,\c d)>0$ if $(r_i,\lambda_g)\in\tilde{\mathcal{R}}(d)$, $\sqrt{3}-1< \omega <1$ and $0\le \c d\le\frac32$.
Now, notice that by \eqref{e:dx1} it follows
\begin{equation}\label{e:ooofs}
\sup_{[0, \gamma]}\sqrt{\Delta(Dg, \alpha)}\le \sqrt{r_iD_g(d-1)}.
\end{equation}
Then, condition \eqref{e:ppoo} implies \eqref{e:c>0} by \eqref{e:tau} and \eqref{e:ooofs}.
\end{proof}

 \begin{remark}
 \label{rem}
 {
 \rm
 We now show that there is a non-empty intersection between the cone $\tilde{\mathcal{R}}(d)$ in Remark \ref{r:RR} and the set of parameters described by Remark \ref{rem:specific1}, for $\sqrt{3}-1< \omega <1$, i.e., for $d>4+2\sqrt{3}\sim 7.46$.
In fact, notice that
 \[
 \frac{1}{\sqrt{3}-1}> \frac{1-\omega}{2+\omega} \ \mbox{ for every } \ 0< \omega <1.
 \]
Then it follows that  $\tilde{\mathcal{R}}(d)\cap\mathcal{T}_g(d)\ne\emptyset$ for  $d>4+2\sqrt{3}$, see Figure \ref{f:triangle_R}. The set $\mathcal{T}_g(d)$ was introduced in Section \ref{s:model}. As a consequence, if $\sqrt{3}-1< \omega <1$, $(r_i,\lambda_g)\in\tilde{\mathcal{R}}(d)\cap\mathcal{T}_g(d)$ and \eqref{e:ppoo} are satisfied, then there are wavefronts to equation \eqref{e:model} satisfying \eqref{e:bvs} having positive speeds.

\begin{figure}[htbp]
    \centering
    \includegraphics[width=8cm]{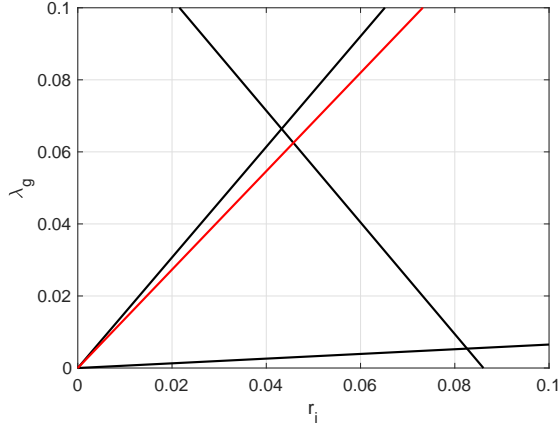}
  \caption{The triangle $\mathcal{T}_g(d)$ (thick black lines) and the cone $\tilde{\mathcal{R}}(d)$, which is bounded from below by the red line and from above by a black line. Here $|C_g|\sqrt{D_g}=6$ and $d=10$.}
    \label{f:triangle_R}
    \end{figure}
 }
 \end{remark}

About the case of {\em negative speeds} we have  the following result.

\begin{theorem}\label{t:modelc<0}
Assume $f$ is strictly concave and
\begin{equation}
\frac{\sqrt{(d-1)\omega(1+\omega)(2-\omega)\left((1+\omega)\mu -(2-\omega)\right)}}{(1+\omega)^2 +\c d(1-\omega^2)-3}> E_g.
\label{e:negspeeds}
\end{equation}
Then either $\mathcal{J}\subset (-\infty,0)$ or $\mathcal{J}=\emptyset$.
\end{theorem}

\begin{proof}
First, we point out that the term $(1+\omega)r_i-(2-\omega)\lambda_g$ under the square root in \eqref{e:negspeeds} is positive because of \eqref{e:restr1-3}. By \eqref{e:coeffs f} we compute
\[
\begin{aligned}
\dot{f}(\alpha)&=-C_gD_g\left(3(\alpha-1)^2-1\right)+C_iD_i(1-\alpha)(3\alpha-1)
\\
&= \frac{C_gD_g\left(3-(1+\omega)^2\right)+C_iD_i(1-\omega^2)}{3}.
\end{aligned}
\]
By \eqref{e:Df2} and \eqref{e:Df3} we deduce
$\dot{D}(\alpha)=3(\alpha-\beta)(D_i-D_g)=-2\omega(D_i-D_g)$,
and, by \eqref{e:gamma}
\[
\begin{aligned}
g(\alpha)&=(r_i+\lambda_g)\alpha(1-\alpha)(\alpha-\gamma)
\\
&=\frac{(r_i+\lambda_g)(2-\omega)(\omega+1)(2-\omega-3\gamma)}{27}
\\
&=\frac{(2-\omega)(\omega+1)\left(-(\omega+1)r_i+(2-\omega)\lambda_g\right)}{27}.
\end{aligned}
\]
Therefore we have
\[
\dot{D}(\alpha)g(\alpha)= \frac{2}{27}(D_i-D_g)\omega(2-\omega)(\omega+1)\left((1+\omega)r_i-(2-\omega)\lambda_g\right).
\]
The proof is concluded by applying \eqref{e:negative speeds}  and noticing that $2\sqrt{2/3}\in(1,2)$.
\end{proof}

\begin{corollary}\label{cor:messi}
Under \eqref{e:restr2}, condition \eqref{e:negspeeds} is satisfied if $d<(5+2\sqrt{3})/2$.
\end{corollary}
\begin{proof}
For $A(\omega):=\omega(1+\omega)(2-\omega)$ and $B(\omega):=(1+\omega)\mu-(2-\omega)$,
condition \eqref{e:negspeeds} is
\begin{align}
\frac{\sqrt{d-1}}{E_g}
\sqrt{A(\omega)B(\omega)}> (1+\omega)^2 +\c d(1-\omega^2)-3=:E(\omega,\c d).
\label{e:ppems}
\end{align}
Condition \eqref{e:negspeeds} is satisfied if the right-hand side of \eqref{e:ppems} is negative.
If $C_i=\c=0$ then this happens if $\omega<\sqrt{3}-1$. In the general case, we notice that
\[
E(\omega, \c d)\le E\left(\omega, \frac32\right)=-\frac{1}{2}\omega^2+2 \omega-\frac{1}{2}:=\phi(\omega).
\]
We have $\phi(0)=-\frac{1}{2}$, $\phi(1)=1$, $\phi$ is an increasing function when $\omega\in(0,1)$, and
$\phi(\omega)=0$ for $\omega\in(0,1)$ iff $\omega=2-\sqrt 3$. Then condition \eqref{e:negspeeds} is satisfied if $\omega<2-\sqrt{3}$, i.e., for $d<(5+2\sqrt{3})/2\sim 4.23$.
\end{proof}

\begin{remark}\label{rem:negative}
{\rm
Let us fix $C_g, D_g$. By Remark \ref{rem:specific1}, for every $\c >0$ and $d>4$ satisfying \eqref{e:restr2} the existence of wavefronts to \eqref{e:model} satisfying \eqref{e:bvs} holds for $(r_i,\lambda_g)$ in the triangle $\mathcal{T}_g(d)$. For $d\in(4,(5+2\sqrt{3})/2)$, every pair $(r_i,\lambda_g)\in\mathcal{T}_g(d)$ provides profiles, and all of them have {\em negative} speeds.
}
\end{remark}

\begin{remark}
\label{r:ext-per}
{\rm
When  $\gamma \to \alpha$, i.e., when $\gamma \to (2-\omega)/3$, we get $\tau(\omega,\c d, \gamma) \to 3E(\omega,\c d)$, for $E$ as in \eqref{e:ppems}. Hence, if $\gamma \sim \alpha$, the condition $\tau < 0$ implies that only wavefronts with negative speeds can agree with \eqref{e:model}-\eqref{e:bvs} (from \eqref{e:ppems}). This implies that the model only supports extinction.
}
\end{remark}
\subsection{A convective term which changes concavity}
\label{ssec:change-conv}
We now consider a convective term $f$ as in \eqref{e:coeffs f} (see also  \eqref{e:Df1}) which changes its concavity in $[0,1]$ and show that also in this case the model \eqref{e:model} can support wavefronts satisfying condition \eqref{e:bvs}. Due to the definition of $f$, a concavity change occurs iff $p\ne0$, and in this case only once, namely at $\frac23 + \frac{q}{3p}$. Moreover, when this occurs, then concavity and convexity are strict.

\begin{lemma}
\label{l:convconc}
Assume that $f$ has an inflection point in $(0,1)$. Then:
\begin{enumerate}[(i)]

\item $f$ is first convex and then concave  if and only if $\c d>\frac32$.

\item $f$ is first concave and then convex if and only if $\c<0$.
\end{enumerate}
\end{lemma}
\begin{proof}
We argue as in the proof of Lemma \ref{lem:conc}. About {\em (i)}, the statement is  equivalent to
$2p+q> 0$ and $-p+q < 0$, i.e., $-2p<q<p$; hence $p>0$, and we conclude by \eqref{e:notation} and $\eqref{e:restr1-1}_1$.

About {\em (ii)}, the statement is  equivalent to $2p+q< 0$ and $-p+q > 0$, i.e., $p<q<-2p$; hence $p<0$ and $\c<0$ by $\eqref{e:restr1-1}_1$.
\end{proof}

To simplify calculations, in the following we only consider the case when $\gamma$, which is the inner zero of $g$ and is given by \eqref{e:gamma}, coincides with the inflection point of $f$; i.e., we assume in the current section (without further mention)
\begin{equation}\label{e:inflection}
\gamma=
\frac 23+\frac{C_gD_g}{3(C_gD_g+C_iD_i)} = \frac{3-2\c d}{3(1-\c d)}.
\end{equation}
Notice that the assumptions $p\ne 0$ and $r_i\ne0$ are equivalent to $\c d\ne 1$ (because of $\eqref{e:restr1-1}_1$) and $\c d\ne \frac 32$, respectively. Then
\begin{equation}\label{e:rilambdag}
r_i=\left(2-\frac{3}{\c d}\right)\lambda_g.
\end{equation}

\subsubsection{The convex-concave case}
We consider a function $f$ which is first convex and then concave, with $\gamma$ as inflection point, see Figure \ref{f:twof}.

%

We recall that we are assuming $\gamma\in(\alpha,\beta)$, see \eqref{e:restr1-3}. We now check the implications of this condition on $\c d$, because $\gamma$ also satisfies \eqref{e:inflection}. By Lemma \eqref{l:convconc}{\em (i)} and  $\eqref{e:restr1-1}_1$ we obtain that $C_iD_i+C_gD_g>0$ and hence $\gamma<\frac23<\beta$ by \eqref{e:inflection}.
On the other hand, the condition $\gamma>\alpha$ is equivalent to $\c d>1+\frac{1}{\omega} >2$ because of \eqref{e:alphabeta} and \eqref{e:inflection}, which strengthens the previous requirement $\c d>\frac32$. Summing up, under the assumptions of the current case, the parameters $\c d$ and $\gamma$ must satisfy the conditions
\begin{equation}\label{e:cd}
\c d>1+\frac{1}{\omega}  \qquad \text{and} \qquad \gamma \in \left(\frac13, \frac23\right).
\end{equation}

We now consider the issue of the existence of profiles. By making use of \eqref{e:deltaphiphi0}, the left-hand side of \eqref{e:ssigma} becomes
\begin{align}
\inf_{[0, \gamma]}\delta(f, \alpha)-\sup_{[\gamma, 1]}\delta(f, \beta) & =\frac{f(\alpha)}{\alpha}-\frac{f(\gamma)-f(\beta)}{\gamma-\beta}
\nonumber
\\
&=(2p+q)(\alpha-\beta-\gamma)+p(\beta^2-\alpha^2+\gamma\beta+\gamma^2)
\nonumber
\\
&=C_gD_g H_1(\omega, \gamma) +C_iD_i H_2(\omega, \gamma)
\nonumber
\\
&=\vert C_g\vert D_g\left(H_1(\omega, \gamma)+\c dH_2(\omega, \gamma)\right),
\label{e:infsup2}
\end{align}
where
\begin{equation*}
H_1(\omega, \gamma):=-\gamma^2-\gamma\left(\frac{\omega-7}{3}\right)+\frac {10}{9} \omega \quad \hbox{ and }  \quad  H_2(\omega, \gamma):=\gamma^2+\gamma\left(\frac{\omega-4}{3}\right)- \frac49\omega.
\end{equation*}
We now investigate the sign of  \eqref{e:infsup2}: its positivity is necessary for \eqref{e:ssigma} to hold. The set $\mathcal{S}$ has been defined in \eqref{e:setS}.

\begin{proposition}
\label{p:signinfsup1}
The quantity in \eqref{e:infsup2} is positive for every $(\omega, \c d) \in \mathcal{S}$.
\end{proposition}
\begin{proof}
We know that $\gamma$, provided by \eqref{e:inflection}, is entirely determined by $\c d$ and that it varies in $(\frac13,\frac23)$ by \eqref{e:cd}. However, to simplify computations, we treat $\gamma$ in the current proof as an independent variable ranging in $(\frac13,\frac23)$.

First, we claim that for $\omega \in (0,1)$ and $\gamma \in (\frac 13, \frac 23)$ we have
\begin{equation}\label{e:estH12}
H_1(\omega, \gamma)>\frac 23 +\omega \quad \hbox{ and }\quad H_2(\omega, \gamma) >-\left(\frac{4+\omega}{6}\right)^2.
\end{equation}
In fact, estimate $\eqref{e:estH12}_1$ follows because the function $\gamma \mapsto H_1(\omega, \gamma)$ is increasing for $\gamma \in (\frac 13, \frac 23)$. Concerning  $\eqref{e:estH12}_2$, we have
$\min_{\gamma \in (\frac13, \frac23)}H_2(\omega, \gamma)=H_2(\omega, \frac{4-\omega}{6})=-(\frac{4+\omega}{6})^2$.

Next, according to \eqref{e:estH12} and since $\c d>0$ we have, for all $\gamma \in (\frac13, \frac 23)$,
\begin{equation}\label{e:stima-H1H2}
H_1(\omega, \gamma)+\c dH_2(\omega, \gamma)>\frac{2+3\omega}{3}-\c d\left(\frac{4+\omega}{6} \right)^2.
\end{equation}
The latter quantity is positive iff $\c d <\frac{12(2+3\omega)}{(4+\omega)^2}
$. By $\eqref{e:cd}_1$ we need $\frac{12(2+3\omega)}{(4+\omega)^2}>1+\frac{1}{\omega}$, and this is equivalent to require $\omega>\omega_0$, where $\omega_0 \sim 0.78$ is the only root of $\omega^3-27\omega^2+16$ in the interval $(0,1)$.
\end{proof}

The following result shows the existence of wavefronts for  equation \eqref{e:model} satisfying \eqref{e:bvs}  when $f$ is first convex and then concave.

\begin{theorem}\label{t:convconc-e}
Assume that $f$ is convex in $[0,\gamma]$, concave in $[\gamma,1]$, and $(\omega, sd)\in \mathcal{S}$.
If
\begin{equation}\label{e:formula}
\frac{\sqrt{(d-1)}}{H_1\left(\omega, \gamma\right)+\c dH_2\left(\omega,  \gamma\right)} < \frac{E_g}{4}
\end{equation}
holds, then equation \eqref{e:model} has wavefronts  satisfying condition \eqref{e:bvs}.
\end{theorem}
\begin{proof}
According to \eqref{e:rilambdag}, Lemma \ref{l:stimaDx} and the fact that $\c d>0$, we have
\begin{align}
2\sup_{[0, \gamma]}\sqrt{\Delta(Dg, \alpha)} &+2\sup_{[\gamma,1]}\sqrt{\Delta(Dg, \beta)}\le \sqrt{D_g}\sqrt{d-1}\left( \sqrt{r_i(2+\omega)} + \sqrt{\lambda_g(1+\omega)}\right)\nonumber\\
= &\sqrt{D_g}\sqrt{\lambda_g(d-1)}\left(\sqrt{\left(2-\frac{3}{\c d}\right)(2+\omega)} +\sqrt{1+\omega}\right) \nonumber\\
\le &\sqrt{D_g}\sqrt{\lambda_g(d-1)}\left(\sqrt{\left(6-\frac{9}{\c d}\right)} +\sqrt{2}\right)
\le 4\sqrt{D_g}\sqrt{\lambda_g(d-1)}.
\end{align}
Now, we assumed $(\omega, \c d) \in \mathcal{S}$ and then $H_1(\omega, \gamma)+\c dH_2(\omega, \gamma)>0$ by Proposition \ref{p:signinfsup1}. Hence, if \eqref{e:formula} is satisfied, then condition \eqref{e:ssigma} holds true by \eqref{e:infsup2} and then \eqref{e:model} has wavefronts satisfying \eqref{e:bvs}.
\end{proof}

\subsubsection{The concave-convex case}
We now assume that $f$ is concave in $[0, \gamma]$ and convex in $[\gamma, 1]$, with $\gamma$ as inflection point, see Figure \ref{f:twof}. Again, we show that \eqref{e:model} admits wavefronts satisfying \eqref{e:bvs} under some conditions.

%

We argue as in the convex-concave case. Lemma \ref{l:convconc}{\em (ii)} implies $\gamma \in (2/3,1)$. Moreover, the condition $\gamma<\beta$ is equivalent to $\c d<1-\frac1\omega<0$ by \eqref{e:alphabeta} and \eqref{e:inflection}. Summing up, the parameters $\c d$ and $\gamma$ must now satisfy the conditions
\begin{equation}\label{e:cd2}
\c d<1-\frac{1}{\omega} \quad \mbox{ and } \quad \gamma \in \left(\frac23, 1\right).
\end{equation}

We now compute the left-hand side of \eqref{e:ssigma}. According to \eqref{e:deltaphiphi0} we have
\begin{align}
\inf_{[0, \gamma]}\delta(f, \alpha)-\sup_{[\gamma, 1]}\delta(f, \beta)&=\frac{f(\gamma)-f(\alpha)}{\gamma-\alpha}-\frac{f(1)-f(\beta)}{1-\beta}\nonumber\\
&=(2p+q)(\alpha-\beta+\gamma-1)+p(\beta^2-\alpha^2-\alpha\gamma-\gamma^2+\beta+1),\nonumber\\
&=C_gD_g \tilde{H}_1(\omega, \gamma) +C_iD_i \tilde{H}_2(\omega, \gamma)\nonumber\\
&=\vert C_g\vert D_g \left(\tilde{H}_1(\omega,\gamma)
 + \c d\tilde{H}_2(\omega,\gamma)\right),
\label{e:infsup3-1}
\end{align}
for
\begin{equation*}
\tilde{H}_1(\omega, \gamma):=\gamma^2-\gamma\left(\frac{\omega+7}{3}\right)+\frac {7}{9} \omega +\frac 43 \quad \hbox{ and }  \quad  \tilde{H}_2(\omega, \gamma):=-\gamma^2+\gamma\left(\frac{\omega+4}{3}\right)- \frac{\omega}{9}- \frac 13.
\end{equation*}

We now discuss the sign of \eqref{e:infsup3-1}; the set $\tilde{\mathcal{S}}$ was defined in \eqref{e:setStilde}.

\begin{proposition}
\label{l:signinfsup3}
The quantity in \eqref{e:infsup3-1} is positive for $(\omega, \c d) \in \tilde{\mathcal{S}}$.
\end{proposition}
\begin{proof}
Notice that, according to \eqref{e:inflection} $\gamma$ depends on $\c d$; however, as in the proof of Proposition \ref{p:signinfsup1}, we treat $\gamma$ as an independent variable ranging in $(0,1)$.

First, we claim that for all $\omega,\gamma \in (0,1)$ we have
\begin{equation}\label{e:sign}
\tilde H_1(\omega, \gamma) > \frac 49 \omega \quad \hbox{ and }\quad
\tilde H_2(\omega,\gamma) <\left(\frac{\omega +2}{6}\right)^2.
\end{equation}
About $\eqref{e:sign}_1$, since $\tilde H_1(\omega, \gamma)$ is a decreasing function for $\gamma \in (-\infty, \frac{\omega+7}{6})$ and $\frac{\omega+7}{6}>1$,  we obtain $\tilde H_1(\omega, \gamma)> \tilde H_1(\omega, 1)=\frac 49 \omega$ for  $\gamma \le 1$. About $\eqref{e:sign}_2$, we have $\max_{\gamma \in \mathbb{R}} \tilde H_2(\omega,\gamma)=\tilde H_2(\omega,\frac{\omega+4}{6})=(\frac{\omega+2}{6})^2$.

Next, according to \eqref{e:sign} and since $\c d<0$ we have
\begin{equation}\label{e:stimatilde}
\tilde{H}_1(\gamma)+\c d\tilde{H}_2(\gamma)>\frac 49 \omega+\c d \left( \frac{\omega +2}{6}\right)^2,
\end{equation}
which is positive when
$\c d > -16\omega/(\omega +2)^2$.
Since we have $\c d < 1-\frac{1}{\omega}$ we need that
\[
1-\frac{1}{\omega}>-\frac{16\omega}{(\omega +2)^2},
\]
which is true when $\omega>\tilde\omega_0$, where $\tilde\omega_0\sim 0.45$ is the only root of $\omega^3+19\omega^2-4$ in $(0,1)$. This completes the proof.
\end{proof}

\begin{theorem}
Assume  that $f$ is concave in $[0,\gamma]$, convex in $[\gamma,1]$ and $(\omega,sd)\in\tilde{\mathcal{S}}$. If
\begin{equation}
\label{e:formula2.2}
\frac{5\sqrt{d-1}}{(1-\omega)\left(\tilde{H_1}(\omega,\gamma)+\c d \tilde{H_2}(\omega,\gamma)\right)} < E_g
\end{equation}
holds, then the equation \eqref{e:model} has wavefronts  satisfying \eqref{e:bvs}.
\end{theorem}
\begin{proof}
As in the proof of Theorem \ref{t:convconc-e}, by taking $\c d < 1-\frac1{\omega}$ into account, we have
\begin{align}
2\sup_{[0, \gamma]}\sqrt{\Delta(Dg, \alpha)}&+2\sup_{[\gamma,1]}\sqrt{\Delta(Dg, \beta)}
\le \sqrt{\lambda_gD_g} \sqrt{d-1}\left(\sqrt{\left(6-\frac{9}{\c d}\right)} +\sqrt{2}\right)
\nonumber
\\
\le & \sqrt{\lambda_gD_g} \sqrt{d-1}\left(\sqrt{\frac{6+3\omega}{1-\omega}} +\sqrt{2}\right)
\le \sqrt{\lambda_gD_g} \sqrt{d-1}\left(\frac{3}{\sqrt{1-\omega}} +\sqrt{2}\right)
\nonumber
\\
\le & \sqrt{\lambda_gD_g} \sqrt{d-1}\frac{3 +\sqrt{2}}{\sqrt{1-\omega}}
\le  5\sqrt{\lambda_g(d-1)}\frac{\sqrt{D_g}}{1-\omega}.
\label{e:estrhs-bis}
\end{align}
Since we assumed $(\omega,\c d)\in\tilde{\mathcal{S}}$, then, by \eqref{e:infsup3-1}, conditions \eqref{e:formula2.2} and \eqref{e:estrhs-bis} imply \eqref{e:ssigma}; according to Corollary \ref{t:CS}, the model  \eqref{e:model} has wavefronts satisfying \eqref{e:bvs}.
\end{proof}

\section*{Acknowledgments}
The authors are members of the {\em Gruppo Nazionale per l'Analisi Matematica, la Probabilit\`{a} e le loro Applicazioni} (GNAMPA) of the {\em Istituto Nazionale di Alta Matematica} (INdAM) and acknowledge financial support from this institution. D.B. was supported by PRIN grant 2020XB3EFL.


\end{document}